\documentclass[11pt]{article} \usepackage[utf8]{inputenc}
\usepackage[T1]{fontenc}
\usepackage{lmodern}
\usepackage[english]{babel}

\usepackage[a4paper,vmargin={3.5cm,3.5cm},hmargin={2.5cm,2.5cm}]{geometry}
\usepackage[font={small,sf}, labelfont={sf,bf}, margin=1cm]{caption}

\usepackage[pdftex,colorlinks=true]{hyperref}
\usepackage[expansion]{microtype}
\usepackage[pdftex]{color,graphicx}
\usepackage{amsmath,amsfonts,amssymb,amsthm,mathrsfs,ulem}
\usepackage{stackrel,framed}
\usepackage{subfig}

\theoremstyle{plain}

\newtheorem{coro}{Corollary}
\newtheorem{theo}[coro]{Theorem}
\newtheorem{prop}[coro]{Proposition}

\newtheorem{lemm}[coro]{Lemma}

\newtheorem{rema}[coro]{Remark}

\newtheorem{defi}[coro]{Definition}

\graphicspath{{images/}}

\newcommand{\kP}{\mathcal{P}}
\newcommand{\kQ}{\mathcal{Q}}

\newcommand{\kT}{\mathcal{T}}

\newcommand{\dloc}{\ensuremath{d_{\mathrm{loc}}}}
\newcommand{\ps}{\oplus}
\newcommand{\ns}{\ominus}

\renewcommand{\leq}{\leqslant}
\renewcommand{\geq}{\geqslant}

\newcommand{\tr}[1][]{\ensuremath{\kT_{#1}}} %trig (type I dans tout le papier)
\newcommand{\trw}[2][]{\ensuremath{\kT_{#1}^{#2}}} %trig a bord simple (argument = longueur du bord)
\newcommand{\bns}[2][]{\ensuremath{\kQ_{#1}^{#2}}} %trig a bord NON simple (argument = longueur du bord)

\newcommand{\gnstr}[1][]{\ensuremath{Q_{#1}}} %Series de Mireille et Olivier
\newcommand{\gtr}[1][]{\ensuremath{Z_{#1}}} %Series a bords simples 
\newcommand{\pB}{\ensuremath{\mathbb{P}_{\mathrm{bol}}}}
\newcommand{\pn}{\ensuremath{\mathbb{P}_n}}
\newcommand{\plim}{\ensuremath{\mathbb{P}_\infty}}

\newcommand{\bol}{\ensuremath{\mathrm{bol}}}
\hypersetup{
    pdfauthor   = {Albenque, Ménard and Schaeffer},%
    pdftitle    = {Ising Random Triangulations}}

\begin{document}
\normalem
\title{\bf Local convergence of large random triangulations coupled with an Ising model}
\author{\textsc{Marie Albenque}\footnote{albenque@lix.polytechnique.fr} \,, \textsc{Laurent Ménard}\footnote{laurent.menard@normalesup.org}
\, and  \textsc{Gilles Schaeffer}\footnote{gilles.schaeffer@lix.polytechnique.fr
}}
\date{}

\maketitle

%%%----------------------------------------------------------------------------------------------------------

\begin{abstract}
We prove the existence of the local weak limit of the measure obtained by sampling random triangulations of size $n$ decorated by an Ising configuration with a weight proportional to the energy of this configuration. To do so, we establish the algebraicity and the asymptotic behaviour of the partition functions of triangulations with spins for any boundary condition. In particular, we show that these partition functions all have the same phase transition at the same critical temperature. Some properties of the limiting object -- called the Infinite Ising Planar Triangulation -- are derived, including the recurrence of the simple random walk at the critical temperature.
% old one : 
% We prove that the measure on triangulations of size $n$ endowed with an Ising configuration biased by the energy of the configuration converges weakly as $n \to \infty$ for the local topology. To do so, for any boundary condition, we establish the algebraicity and the asymptotic behaviour of the partition functions of triangulations with spins weighted by their energy. In particular, we show that these partition functions all have the same phase transition at the same critical temperature. Some properties of the limiting object -- called the Infinite Ising Planar Triangulation -- are derived, including the recurrence of the simple random walk at the critical temperature.

\end{abstract}

\bigskip

% \noindent{\bf MSC 2010 Classification:} Des codes ...\\
% \noindent{\bf Keywords:} Des mots ...

% \bigskip

%%%----------------------------------------------------------------------------------------------------------
{\let\thefootnote\relax\footnote{2010, \emph{Mathematics subject classification} 05A15, % 	Exact enumeration problems, generating functions
05A16, % 	Asymptotic enumeration05C12, %Distance in graphs
05C30, % 	Enumeration in graph theory
60C05, %  	Combinatorial probability
60D05, %  	Geometric probability and stochastic geometry
60K35, %  	Interacting random processes; statistical mechanics type models; percolation theory
82B44 %	Disordered systems (random Ising models, random Schrödinger operators, etc.)
}}

\section{Introduction}
In 2003, in order to define a model of \emph{generic} planar geometry, Angel and Schramm studied the limit of uniform triangulations on the sphere, \cite{AngelSchramm}. They proved that this model of random maps converges for the Benjamini-Schramm topology defined in \cite{BS}, and also called the \emph{local topology}. The limiting object is a probability distribution on infinite triangulations, now known as the UIPT, for Uniform Infinite Planar Triangulation. Soon after, Angel~\cite{AngelPerco} studied some properties of the UIPT. He established that the volume of the balls of the UIPT of radius $R$ scales as $R^4$ and that the site-percolation threshold is 1/2. 

Similar results (but with quite different proofs) were then obtained for quadrangulations by Chassaing and Durhuus \cite{ChassaingDurhuus} and Krikun~\cite{KrikunQuad}. Since then, the local limit of random maps has become an active area of research. The UIPT is now a well-understood object: the simple random walk on the UIPT is known to be recurrent~\cite{GGN}, precise estimates about the volume and the perimeter of the balls of radius $r$ are available~\cite{CLGpeel}, geodesic rays are known to share infinitely many cutpoints \cite{CuMe} and percolation is fairly well understood \cite{AngelPerco,AngelUIHPT,BeCuMie,BeHoSun,CuKo,GMSS}. We refer to the recent survey by Le Gall~\cite{LGsurvey} or the lecture notes by Miermont \cite{MieSF} for nice entry points to this field, and to~\cite{autopromo} for a survey of the earlier combinatorial literature on random maps.

The results cited above deal with models of maps that fall in the same ``universality class'', identified in the physics literature as the class of ``pure 2d quantum gravity'': the generating series all admit the same critical exponent and the volume of the balls of the local limits of several of those models of random maps are known to grow as $R^4$. To capture this universal behaviour, a good framework is to consider scaling limits of random maps (of finite or infinite size) in the Gromov Hausdorff topology. Indeed, for a wide variety of models the scaling limit exists and is either the Brownian map \cite{Ab,AA,BJM,LGbm,MaMo,Miebm} or the Brownian plane \cite{BMR,CLGbp}.

To escape this pure gravity behaviour, physicists have long ago understood that one should ``couple gravity with matter'', that is, consider
models of random maps endowed with a statistical physics model. From a combinatorial point of view, evidence for the existence of other
universality classes were first given by constructing models, like tree-rooted maps or triangulations endowed with Ising configurations, whose generating series exhibit a different asymptotic behaviour at criticality. One of the first such results, and the most relevant for our work,
appears in~\cite{BoulatovKazakov}, where Boulatov and Kazakov initiated the study of Ising models on random triangulations (following some earlier work by Kazakov~\cite{Kaza86}). They established the existence of a phrase transition, the critical value of the model and the corresponding critical exponents. Their result is based on the expression of the generating series of the model as a matrix integral and the use of orthogonal polynomial methods. Their result was later rederived via bijections with trees by Bousquet-Mélou and the third author \cite{BMS} and by Bouttier, di Francesco and Guitter \cite{BdFGIsing}, and more recently via a tour-de-force in generatingfunctionology by Bernardi and Bousquet-Mélou \cite{BernardiBousquet} building on a seminal series of papers by Tutte on the enumeration of colored maps, synthesized in~\cite{Tutte}.

\subsubsection*{Main results}
The aim of this paper is to build on these latter ideas to prove the local convergence of random triangulations endowed with Ising configurations. To state our main result, let us first introduce some terminology. Precise definitions will be given in Section~\ref{sub:def}. For $T$ a rooted finite triangulation of the sphere, a \emph{spin configuration} on (the vertices of) $T$ is an application $\sigma : V(T)\rightarrow \{\ns,\ps\}$. We denote by $\tr[f]$ the set of finite triangulations endowed with a spin configuration. For $(T,\sigma) \in \mathcal{T}_f$, we write $m(T,\sigma)$ for its number of monochromatic edges. Then, for $n\in \mathbb{N}$ and $\nu>0$, let $\mathbb P_n^\nu$ be the probability distribution supported on elements of $\tr[f]$ with $3n$ edges, defined by: 
\begin{equation} \label{eq:defPn}
\mathbb P_n^\nu \left( T, \sigma \right) \propto \nu^{m(T,\sigma)}\mathbf{1}_{\{ |T| = 3n \}}.
\end{equation}
Writing $\nu=\exp(-2\beta)$, this is the probability distribution obtained when sampling a triangulation with $3n$ edges together with a spin configuration on its vertices with a probability proportional to the energy in the Ising model, defined by $\mathrm{exp}(-\beta\sum_{(v,v')\in E(T)}\sigma(v)\sigma(v'))$. In particular, the model is \emph{ferromagnetic} for $\nu>1$ and \emph{antiferromagnetic} for $0< \nu < 1$. The case $\nu =1$ corresponds to uniform triangulations.

Following Benjamini and Schramm~\cite{BS}, we equip the set $\tr[f]$ with the local distance $\dloc$. For $(T,\sigma),(T',\sigma')$ in $\tr[f]$, set:
\begin{equation}
\label{eq:dloc}
\dloc((T,\sigma),(T',\sigma')) = (1 + \sup \{ R \geq 0 : B_R (T,\sigma) = B_R(T',\sigma')\})^{-1},
\end{equation}
where $B_R(T,\sigma)$ is the submap of $T$ composed by its faces having at least one vertex at distance smaller than $R$ from its root vertex, with the corresponding spins. The only difference with the usual setting is the presence of spins on the vertices and, in addition of the equality of the underlying maps, we require equality of spins as well. 

The closure $(\tr,\dloc)$ of the metric space $(\tr[f],\dloc)$ is a Polish space and elements of $\tr \setminus \tr[f]$ are called infinite triangulations with spins. The topology induced by $\dloc$ is called the \emph{local topology}. As it is often the case with local limits of planar maps, we will be especially interested in the element of $\tr$ that are \emph{one-ended}, that is the infinite triangulations $(T,\sigma) \in \tr$ for which $(T,\sigma) \setminus B_R(T,\sigma)$ has a unique infinite connected component for every $R$.

Our main theorem is the following result: 
\begin{theo}
\label{th:localCV}
For every $\nu>0$, the sequence of probability measures $\pn^\nu$ converges weakly for the local topology to a limiting probability measure $\plim^\nu$ supported on one-ended infinite triangulations endowed with a spin configuration. 

We call a random triangulation distributed according to this limiting law the \emph{Infinite Ising Planar Triangulation with parameter $\nu$} or \emph{$\nu$-IIPT}.
\end{theo}

Our approach to prove this convergence result is akin to Angel and Schramm's initial approach for the UIPT: in particular it requires precise information about the asymptotic behaviour of the partition function of large Ising triangulations, with an arbitrary fixed boundary condition, see Theorem~\ref{theo:combimain}. This result, which does not follow from earlier results \cite{BoulatovKazakov,BernardiBousquet,BMS}, constitutes a significant part of this work and is of independent interest. One of the main technical challenges to obtain this result is to solve an equation with two catalytic variables. This is done in Theorem~\ref{th:Aplus} using Tutte's invariants method, following the presentation of~\cite{BernardiBousquet}.

As expected, these partition functions all share the same asymptotic behaviour, which presents a phase transition for $\nu$ equal to $\nu_c := 1+\sqrt{7}/7$. This critical value already appeared in \cite{BoulatovKazakov,BernardiBousquet,BMS}, and we call \emph{critical IIPT} the corresponding limiting object.
The study of this critical IIPT is the main motivation for this work, since, as mentioned above, it is believed to belong to a different class of universality than the UIPT. However, these two models  share some common features, as illustrated by the following theorem:
\begin{theo}
\label{th:recurrence}
The simple random walk on the critical IIPT is almost surely recurrent.
\end{theo}
Our strategy to prove this result does not rely on the specificity of $\nu_c$, but requires a numerical estimate which prevents us from extending this result to a generic $\nu$. However, the same proof would work for any fixed $\nu$ between 0.3 and 2 (see Remark \ref{rem:nurec}) and we conjecture that the IIPT is recurrent for every value of $\nu$.

Finally, as a byproduct of the proof of Theorem~\ref{th:localCV}, we prove a spatial Markov property for the $\nu$-IIPTs (Proposition~\ref{prop:spatialIIPT}) and some of its consequences. We also provide a new tightness argument (see Lemma \ref{lem:root-tight}) that seems simple enough to be adapted to other models since it does not require explicit computations as was the case in previous works.

\subsubsection*{Connection with other works}
Our results should be compared to the recent preprint of Chen and Turunen~\cite{ChenTurunen} where they consider random triangulations with spins on their faces, at a critical parameter similar to our $\nu_c$ and with Dobrushin boundary conditions (i.e. with a boundary formed by two monochromatic arcs, similarly as in Figure~\ref{fig:tutteAS}). In the first part of their paper, the authors compute explicitly the partition function of this model by solving its Tutte's equation, obtaining a result comparable to Theorem \ref{th:Aplus}. While their proof also relies on the elimination of one of their two catalytic variables, it does not use Tutte's invariant like ours. However, as was explained to us by Chen, their algebraicity result and Theorem~\ref{th:Aplus} are equivalent and can be deduced from one another by a clever argument based on the relation between the Tutte polynomial of a
planar map and that of its dual.

In the second part of their paper, Chen and Turunen show that their model has a local limit in distribution when the two components of the Dobrushin boundary tend to infinity one after the other. The fact that they consider these particular boundary conditions allow them to make explicit computations on Boltzmann triangulations and to construct explicitly the local limit using the peeling process along an Ising interface. They also derive some properties of this interface.

\medskip

At the discrete level the Ising model is closely related via spin
cluster interfaces to the $O(n)$ model: this latter model has been studied
on triangulations or bipartite Boltzmann maps via a gasket decomposition approach 
in a series of papers \cite{BBGa,BBGc,BBGb,BBD,BuOn,CCM}, revealing a remarkable connection with the stable maps of \cite{LGM}. In particular this approach allows to identify a dense phase, a dilute phase and a generic phase for the loop configuration. We believe that our approach is suitable to study the geometry of the spin clusters of the Ising model and might shed some additional light on this connection with stable maps. We plan to return to this question soon in a sequel of the
present paper.

\medskip
Let us end this introduction by mentioning the conjectured links between models of decorated maps and Liouville Quantum Gravity (LQG), which is a one-parameter family of random measures on the sphere \cite{DS}. Physicists believe that most models of decorated maps converge to the LQG for an appropriate value of the parameter. In particular, the Ising model should converge to the $\sqrt{3}$-LQG. 

Such a convergence has been established in the case of ``pure quantum gravity'', corresponding to uniform planar maps and $\gamma=\sqrt{8/3}$, in the impressive series of papers by Miller and Sheffield~\cite{MSa,MSb,MSc}. Obtaining such a result for a model of decorated maps outside the pure-gravity class seems out of reach for the moment. However -- building on the so-called mating-of-trees approach initiated by Sheffield~\cite{She} and which has allowed to obtain various local convergence results for models of decorated maps (see e.g.\cite{C,BLR,GMSI19,GSII17,GSIII15}) -- Gwynne, Holden and Sun~\cite{GHS} managed to prove that for some models of decorated maps, including the spanning-tree decorated maps, bipolar oriented maps and Schnyder wood decorated maps, the volume growth of balls in their local limit is given by the ``fractal dimension'' $d_\gamma$, for the conjectured limiting $\gamma$-LQG (see also \cite{GHSSurvey} for a recent survey on this topic by the same authors).

The value of $d_{\gamma}$ is only known in the pure gravity case and $d_{\sqrt{8/3}}=4$. For other values of $\gamma$, only bounds are available. As of today, the best ones have been established by Ding and Gwynne in \cite{DG}. Except when $\gamma$ is close to 0, these bounds are compatible with Watabiki's famous prediction for $d_{\gamma}$~\cite{W}:
\[
	d_\gamma^{Wat} = 1 + \frac{\gamma^2}{4}+\frac{1}{4}\sqrt{(4+\gamma^2)^2+16\gamma^2}.
\]

As far as we understand, the Ising model does not fall into the scope of this mating-of-trees approach and so far, we are not able to derive information on the volume growth of balls in the IIPT. For $\gamma=\sqrt{3}$, Watabiki's prediction gives $d_{\sqrt{3}}^{Wat}=\dfrac{7+\sqrt{97}}{4}\approx 4.212$ and the bounds of Ding and Gwynne give:
\[
	4.189 \approx \frac{7+\sqrt{31}}{3}\leq d_{\sqrt{3}} \leq 3\sqrt{2}\approx 4.243.
\] 
If we believe in the connection between the critical IIPT and $\sqrt{3}-$LQG, this is a strong indication that its volume growth should be bigger than 4. We hope that the present work will provide material for the rigorous study of metric properties of two-dimensional quantum gravity coupled with matter.
\bigskip

\noindent \textbf{Acknowledgments:}
We thank Linxiao Chen and Jérémie Bouttier for insightful discussions, and for sharing the progress of~\cite{ChenTurunen} while the paper was still in preparation. We also warmly thank Mireille Bousquet-Mélou, who patiently answered all our questions about Tutte's invariant method and even shared some of her tricks.

We also wish to thank an anonymous referee, whose many comments improved significantly this article.

This work was supported by the grant ANR-14-CE25-0014 (ANR GRAAL), the grant ANR-16-CE40-0009-01 (ANR GATO) and the Labex MME-DII (ANR11-LBX-0023-01).

\tableofcontents

\section{Enumerative results}

\subsection{Triangulations with spins: definitions and generating series}\label{sub:def}
%\subsection{Triangulations with spins: definitions and main algebraicity result}\label{sub:def}
A \emph{planar map} is the embedding of a planar graph in the sphere considered up to sphere homeomorphisms preserving its orientation. Maps 
 are \emph{rooted}, meaning that one edge is distinguished and oriented. This edge is called the \emph{root edge}, its tail the \emph{root vertex} and the face on its right the \emph{root face}. A \emph{triangulation} is a planar map in which all the faces have degree 3. Note that loops and multiple edges are allowed, so that, in the terminology of Angel and Schramm \cite{AngelSchramm}, we consider \emph{type I-triangulations}.

More generally, a \emph{triangulation with a boundary} is a planar map in which all faces have degree $3$ except for the root face (whose boundary may not be simple) and 
a \emph{triangulation of the $p$-gon} is a triangulation whose root face is bounded by a simple cycle of $p$ edges. Occasionally, we will also consider \emph{triangulations with holes}, which are planar maps such that every face has degree $3$, except for a given number of special faces enclosed by simple paths that will be called holes.
The \emph{size} of a planar map $M$ is its number of edges and is denoted by $|M|$. 

\bigskip

The maps we consider are always endowed with a \emph{spin configuration}: a given map $M$ comes with an application $\sigma$ from the set $V(M)$ of its vertices to the set $\{\ps,\ns\}$. An edge $\{u,v\}$ of $M$ is called \emph{monochromatic} if $\sigma(u)=\sigma(v)$ and \emph{frustrated} otherwise. The number of monochromatic edges of $M$ is denoted by $m(M)$. 

Let $p$ be a fixed positive integer and $\omega=\omega_1\cdots \omega_p$ be a word of length $p$ on the alphabet $\{\ps,\ns\}$. The set of triangulations of a $p$-gon of size $n$ is denoted by $\trw[n]{p}$ (boundary edges are counted). Likewise, the set of finite triangulations of the $p$-gon is denoted by $\kT_f^p$. Moreover, we write $\kT^{\omega}_f$ for the subset of $\kT^p_f$ consisting of all triangulations of the $p$-gon endowed with a spin configuration such that the word on $\{\ps,\ns\}$ obtained by listing the spins of the vertices incident to the root face, starting with the target of the root edge, is equal to $\omega$ (see Figure~\ref{fig:boundary}).
\begin{figure}[t]
\begin{center}
\includegraphics[width=\textwidth]{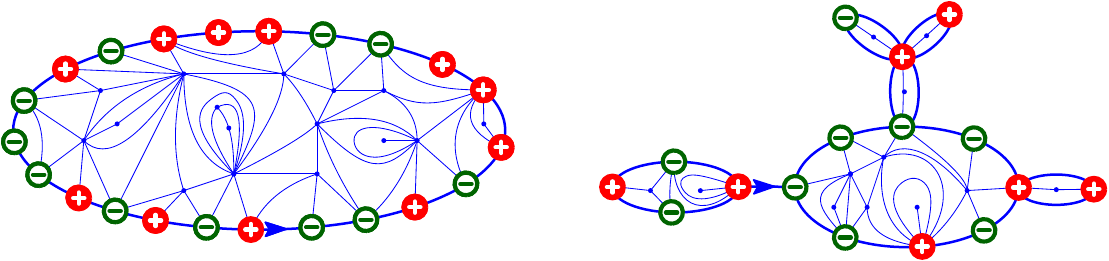}
\caption{\label{fig:boundary}Two triangulations with boundary condition $\omega = \ns \ns \ps \ns \ps \ps \ps \ns \ns \ps \ps \ps \ns \ps \ns \ns \ns \ps \ns \ps \ns \ps$. The one on the left has a simple boundary whereas the one on the right has not.}
\end{center}
\end{figure}

\medskip 

We now introduce the generating series that will play a central role in this paper and are the subject of our main algebraicity theorem.
For any positive integer $p$, the generating series of triangulations of a $p$-gon endowed with an Ising model with parameter $\nu$ is defined as: 
\[
Z_p(\nu,t) = \sum_{T \in \kT_f^p}t^{|T|}\nu^{m(T)}.
\]
For every fixed word $\omega \in \{\ps,\ns\}^p$, we also set
\[
\quad Z_{\omega}(\nu,t) = \sum_{T \in \kT_f^{\omega}}t^{|T|}\nu^{m(T)}.
\]
In particular, the generating series of triangulations of a $p$-gon with positive boundary conditions is given by: 
\[
\quad Z_{\ps^p}(\nu,t) = \sum_{T\in \kT_f^{\ps^p}}t^{|T|}\nu^{m(T)}
\]
where $\ps^p$ denotes the word made of $p$ times the letter $\ps$.

\bigskip

To normalize the probability $\mathbb P_n^\nu$ defined by \eqref{eq:defPn} in the introduction, we consider $\mathcal Z (\nu,t)$ the generating series of the triangulations of the sphere. It is linked to the generating series of the 1-gon and of the 2-gon by the following relation:
\[
\mathcal Z (\nu,t) =  \sum_{(T,\sigma) \in \mathcal T_f} \nu^{m(T,\sigma)} t^{|T|} = \frac{2}{t} \left( \frac{Z_{\ps\ps} (\nu,t)}{\nu} +  Z_{\ps\ns}(\nu,t) + \frac{Z_{\ps}^2(\nu,t)}{\nu}  \right).
\]
Indeed, if the root edge of a triangulation of the sphere is not a loop, by opening it we obtain a triangulation of the $2$-gon giving the first two terms in the sum (we divide $Z_{\ps\ps}$ by $\nu$ in order to count the root edge as a monochromatic edge only once). On the other hand, if the root edge is a loop, we can decompose the triangulation into a pair of triangulations of the $1$-gon giving the last term in the sum. In both cases, the factor $2/t$ is here to count the root edge only once and to take into account the fact that the root vertex can have spin $\ns$ (obviously $Z_\ps = Z_\ns$, $Z_{\ps\ps} = Z_{\ns\ns}$ and $Z_{\ps \ns} = Z_{\ns\ps}$), see Figure~\ref{fig:RootTransform}.
\begin{figure}[!ht]
\begin{center}
\includegraphics[width=0.9\textwidth]{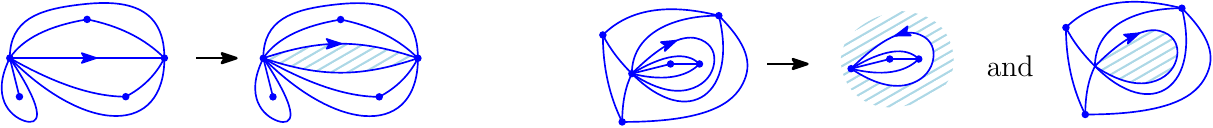}
\caption{\label{fig:RootTransform}How to transform a triangulation of the sphere into a triangulation of the 2-gon (left) or into two triangulations of the 1-gon (right). The $p$-gon is shaded.}
\end{center}
\end{figure}
\medskip

\subsection{Definition of Ising-algebraicity and main algebraicity result}\label{sub:IsingAlgebraic}
Since the number of edges of a triangulation of a $p$-gon is congruent to $-p$ modulo 3 (each triangular face has $3$ half edges), the series $t^pZ_p$ (or $t^pZ_\omega$ if $\omega$ has length $p$) can also be seen as series in the variable $t^3$ that counts the vertices of the triangulation (minus $1$, this is a direct consequence of Euler's formula). The different generating series introduced in the previous section all share common features. In particular, they all have the same radius of convergence. This will be proven later but, since we need the value of this common radius of convergence to state our results, let us define it now. This quantity $\rho_\nu$ satisfies: 
\begin{align*}
  P_2(\nu, \rho_\nu)&= 0 \hbox{ for } 0< \nu \le
  \nu_c:=1+ 1/\sqrt 7,\\
P_1(\nu, \rho_\nu)&=0 \hbox{ for } \nu_c \le \nu,
\end{align*}
where $P_1$ and $P_2$ are the following two polynomials:
\begin{align}
 P_1(\nu,\rho) &= 
 131072\,{\rho}^{3}{\nu}^{9}-192\,{\nu}^{6} \left( 3\,\nu+5 \right) 
 \left( \nu-1 \right)  \left( 3\,\nu-11 \right) {\rho}^{2}
\notag\\
& \quad \quad-48\,{\nu}^{3} \left( \nu-1 \right) ^{2}\rho+ \left( \nu-1 \right)  \left( 4\,{\nu
}^{2}-8\,\nu-23 \right),\label{eq:P1}\\
P_2(\nu,\rho) & = 
27648\,{\rho}^{2}{\nu}^{4}+864\,\nu\, \left( \nu-1 \right)  \left( {
\nu}^{2}-2\,\nu-1 \right) \rho+ \left( 7\,{\nu}^{2}-14\,\nu-9 \right) 
 \left( \nu-2 \right) ^{2}.\label{eq:P2}
\end{align}

\begin{figure}[!ht]
\begin{center}
\includegraphics[width=0.9\textwidth]{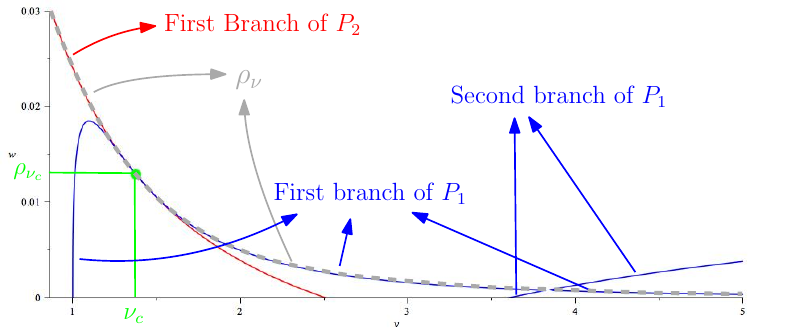}
\caption{\label{fig:rho}Positive branches of the polynomials $P_1$ and $P_2$ and definition of $\rho_\nu$ (represented in fat dashed grey edges).}
\end{center}
\end{figure}

To properly define $\rho_\nu$ as a function of $\nu$, we have to specify which branch of $P_1$ or of $P_2$ to consider. The situation, which is illustrated in Figure \ref{fig:rho} and detailed in the proof of Proposition \ref{prop:asyQ1}, is the following. The polynomial $P_2$ has real roots only for $\nu \in (0,3]$. Of its two branches, only one takes positive values. This branch will be called the first branch of $P_2$, it is given by
\[
w_2 (\nu) = \frac{(1+\nu)(3-\nu) \sqrt{3(1+\nu)(3-\nu)} - 9 (\nu-1)(\nu^2-2\nu -1)}{576 \nu^3},
\]
where $\sqrt{\cdot}$ denotes the principal value of the square root. As a function of $\nu$,  it is decreasing, continuous and positive for $\nu \in (0,\nu_c]$.

The polynomial $P_1$ has a unique real root for $\nu<3$ and three real roots for $\nu \geq3$. Among these real roots, two can take positive values for $\nu > 0$ and meet at $\nu = 1 + 2 \sqrt{2}$. We define $w_1(\nu)$ as the only real root of $P_1$ for $\nu < 3$, its larger positive root for $\nu \in [3,1+2\sqrt{2}]$, and its smaller positive root for $\nu \geq 1+2\sqrt 2$. This defines a branch of $P_1$ that we call the first branch of $P_1$. As a function of $\nu$, it is decreasing, continuous and positive for $\nu \geq \nu_c$. In addition we have $w_1(\nu_c) = w_2(\nu_c)$.

The function $\rho_\nu$ is then defined as follows:
\begin{defi}\label{def:rho}
Set $\nu_c = 1 + 1/\sqrt 7$.
Let $w_2$ be the unique branch of $P_2$ that is positive for $\nu \in (0,\nu_c]$ and $w_1$ be the unique branch of $P_1$ that is positive decreasing for $\nu \geq \nu_c$.
For every $\nu > 0$ we set:
\[
\rho_\nu =
\begin{cases}
w_2(\nu) & \text{ for $\nu \in (0,\nu_c]$,}\\
w_1(\nu) & \text{ for $\nu > \nu_c$.}
\end{cases}
\]
This defines a continuous and decreasing bijection from $(0,+\infty)$ onto $(0,+ \infty)$ (see Figure \ref{fig:rho} for an illustration). The value at $\nu_c$ is
\[
\rho_{\nu_c}
= \frac{25 \,\sqrt{7}-55}{864}.
\]
\end{defi}
The following property is going to be ubiquitous in the rest of the paper:
\begin{framed}
\begin{defi}
A generating series $S(\nu,t)$ is said to be \emph{Ising-algebraic} (with parameters $A$, $B$ and $C$), if the following conditions hold.
\begin{enumerate}
\item For any positive value of $\nu$,  the generating series $S$, seen as a series in $t^3$, is algebraic and $\rho_\nu = (t_{\nu})^3$ is its unique dominant singularity.

\item The series $S$ satisfies the following singular behaviour: there exist non-zero constants $A (\nu)$, $B (\nu)$ and $C (\nu)$ such that:
\begin{itemize}
\item For $\nu \neq \nu_c$, the critical behaviour of $S(\nu,t)$ is  the standard behaviour of planar maps series, with an exponent $3/2$.
$$
S(\nu,t)=A(\nu) +B(\nu) \cdot \big(1-(t/t_\nu)^3\big)+C(\nu) \cdot \big(1-(t/t_\nu\big)^3)^{3/2}+o\Big(\big(1-(t/t_\nu)^3\big)^{3/2}\Big).
$$
\item But, at $\nu=\nu_c$, the nature of the singularity changes and:
$$
S(\nu_c,t)=A(\nu_c) +B(\nu_c) \cdot \big(1-(t/t_{\nu_c})^3\big)+C(\nu_c) \cdot \big(1-(t/t_{\nu_c})^3\big)^{4/3}+o\Big(\big(1-(t/t_{\nu_c})^3\big)^{4/3}\Big).
$$
\end{itemize}
\end{enumerate}
\end{defi}
\end{framed}
Ising-algebraic series all share the same asymptotic behaviour:
\begin{prop}\label{prop:dvIsing}
If $S(\nu,t)$ is Ising-algebraic with parameters $A$, $B$ and $C$, then for any $\nu>0$, we have, as $n \to \infty$: 
\begin{equation}\label{eq:dvIsing}
[t^{3n}]S(\nu,t)\sim
\begin{cases}
k(\nu) \cdot t_\nu^{-3n} \, n^{-5/2} &\text{ if }\nu\neq \nu_c,\\
k(\nu_c) \cdot t_{\nu_c}^{-3n} \, n^{-7/3} &\text{ if }\nu=\nu_c,
\end{cases}
\end{equation}
where 
\[
	k(\nu)=\frac{3C(\nu)}{4\sqrt{\pi}}\text{  if }\nu\neq \nu_c\quad \text{and}\quad k (\nu_c)=\dfrac{C(\nu_c)}{\Gamma(-4/3)}.
\]

\end{prop}
\begin{proof}
This is a direct consequence of the general transfer theorem \cite[Thm VI.3, p.390]{FS}. 
\end{proof}
Our main algebraicity result is the following theorem:
\begin{theo}
\label{theo:combimain}
Denote by $\{\ps,\ns\}^+$ the set of finite nonempty words on the letters $\ps$ and $\ns$.
For any $\omega \in \{\ps,\ns\}^+$, the series $t^{|\omega|}Z_{\omega}(\nu,t)$ is Ising-algebraic with some parameters $A_\omega$, $B_\omega$ and $C_\omega$.
In particular, for any $\nu>0$, there exists $\kappa_\omega(\nu)\in \mathbb{R}_{>0}$ such that, as $n \to \infty$,
\begin{equation}\label{def:kappaomega}
[t^{3n-|\omega|}]Z_{\omega}(\nu,t)\sim
\begin{cases}
\kappa_{\omega}(\nu) \cdot t_\nu^{-3n} \, n^{-5/2} &\text{ if }\nu\neq \nu_c,\\
\kappa_{\omega}(\nu_c) \cdot t_{\nu_c}^{-3n} \, n^{-7/3} &\text{ if }\nu=\nu_c.
\end{cases}
\end{equation}\label{def:kappasphere}
Similarly, for triangulations of the sphere, we have as $n \to \infty$: 
\begin{equation}
[t^{3n}]\mathcal Z (\nu, t)\sim
\begin{cases}
\kappa(\nu) \cdot t_\nu^{-3n} \, n^{-5/2} &\text{ if }\nu\neq \nu_c,\\
\kappa(\nu_c) \cdot t_{\nu_c}^{-3n} \, n^{-7/3} &\text{ if }\nu=\nu_c,
\end{cases}
\end{equation}
with
\[
\kappa(\nu) = \frac{2} {t_\nu^3} \left(\frac{\kappa_{\ps\ps}}{\nu} + \kappa_{\ps\ns} + \frac{\kappa_\ps \, Z_\ps (t_\nu)}{\nu}\right).
\]
\end{theo}

\bigskip

\noindent{\bf Main steps of the proof of Theorem~\ref{theo:combimain}}\\
The rest of this section is devoted to the proof of Theorem~\ref{theo:combimain}. First we recall in Section~\ref{sub:nonsimple} the result of Bernardi and Bousquet-Mélou~\cite{BernardiBousquet} about triangulations with a (non-simple) boundary of size 1 or 3. We then show how a squeeze lemma-type argument allows to extend their result to various models of triangulations provided that algebraicity is proved. Then, the main piece of work is to prove that the generating series of triangulations of a $p$-gon with positive boundary conditions are algebraic, see Section~\ref{sub:plus}. Finally, a double induction on the length of the boundary and on the number of $\ps$ on the boundary allows to conclude the proof, see Section~\ref{sub:anycondition}. 

\subsection{Enumerative results for triangulations with a non simple boundary}
\label{sub:nonsimple}
\subsubsection{Generating series of triangulations with a boundary, following \cite{BernardiBousquet}}
Let $\kQ$ denote the set of triangulations with a boundary (not
necessarily simple), and $\kQ^{p}$ denote the subset of these triangulations
with boundary length equal to $p$. Following~\cite{BernardiBousquet}, we define:
\[
\gnstr[p](\nu,t) = \frac{1}{2}\sum_{M\in \bns{p}}t^{|M|} \nu^{m(M)}, 
\]
%\[
%\gnstr[p](\nu,t) =\frac{1}{2}\sum_{n \geq 0} t^n \sum_{M\in \bns[n]{p}} \nu^{m(M)} 
%= \frac{1}{2}\sum_{M\in \bns{p}}t^{|M|} \nu^{m(M)}, 
%\]
and let 
\[
\gnstr(y)=\sum_{p\geq1}y^p\gnstr[p](\nu,t).
\]
Explicit expressions for $Q_1$ and $Q_3$ have been established by Bernardi and Bousquet-Mélou:
\begin{theo}[Theorem 23 of~\cite{BernardiBousquet}]\label{th:gsBoundary}
Define $U\equiv U(\nu,t)$ as the unique formal power series in $t^3$ having constant term $0$ and satisfying
\begin{equation} \label{eq:t3U}
t^3= U \frac{\Big((1+\nu)U-2 \Big) \Big( 8\nu (\nu+1)^2 U^3 -(11 \nu + 13)(\nu +1) U^2 +2 (\nu +3)(2\nu+1)U -4 \nu \Big)}{32\nu^3 (1-2U)^2}.
\end{equation}
Then, there exist explicit polynomials $\widetilde{R_1}$ and $\widetilde{R_{3}}$ such that:
\begin{equation}\label{eq:expQ1Q3}
t^3\cdot tQ_1 = \nu t^3\cdot t^2Q_2= \frac{\widetilde{R_1}(U,\nu)}{\nu^4 (1-2U)^2}\quad \text{ and }\quad
t^3\cdot t^3Q_3= \frac{\widetilde{R_{3}}(U,\nu)}{\nu^6 (1-2U)^4}.
\end{equation}
\end{theo}

\begin{rema}
Theorem 23 of \cite{BernardiBousquet} gives a different parametrization than the one given in Theorem \ref{th:gsBoundary}. The two are linked with the simple change of variables (which already appears in~\cite{BernardiBousquet}):
\[
U = \frac{1}{2} \left( 1 + \frac{1-\nu}{1+\nu} S\right).
\]
With this change of variables, the values of $\widetilde{R_1}$ and $\widetilde{R_3}$ are given by
\begin{align*}
\widetilde{R_1} (U,\nu) &= \frac{\nu+1}{64} \, U^2 \, (U(\nu+1) - 2) \, \left( 6 (\nu +1)^2 U^3 - 8 (\nu+1)^2 U^2 + (3\nu+4)(\nu+1) U -2 \nu \right),\\
\widetilde{R_3} (U,\nu) &=\frac{-1}{2048} \, U^3 \, (U(\nu+1) - 2) \\
&\quad \times \left(
256 (\nu+1)^5 U^6-(16 (35 \nu+93))(\nu+1)^4 U^5+(491 \nu^2+2938\nu+2715)(\nu+1)^3 U^4 \right.\\
&\qquad -(2 (100 \nu^3+1173 \nu^2+2098 \nu+1237)) (\nu+1)^2 U^3 \\
& \qquad +(4(\nu+1))(8\nu^4+218\nu^3+637\nu^2+708\nu+285)U^2 \\
& \qquad \left. +(-216-672\nu^3-128\nu^4-1176\nu^2-880\nu)U+32\nu (3+3\nu+2\nu^2)
\right)
\end{align*}
\end{rema}

These explicit parametrizations allow to study the singularities of the series $Q_1$ and $Q_3$ as is partly done in \cite{BernardiBousquet} but without a full proof. One of the main missing steps is a complete study of the singular behaviour of $U$. This is the purpose of the following lemma:

\begin{lemm} \label{lem:etU}
For every fixed $\nu >0$, the series $U$, defined by \eqref{eq:t3U} and seen as a series in $t^3$, has nonnegative coefficients and radius of convergence $\rho_\nu$. In addition, it is convergent at $\rho_\nu$, has a unique dominant singularity (at $\rho_\nu$), and has the following singular behaviour: there exists a positive explicit constant $ \aleph(\nu)$ such that
\[
U(\nu,t^3) =
\begin{cases}
U(\nu,\rho_\nu) - \aleph (\nu) \cdot (1-t^3/\rho_\nu)^{1/2}+o\big((1-t^3/\rho_\nu)^{1/2}\big) & \text{ for $\nu \neq \nu_c$, }\\
U(\nu_c,\rho_{\nu_c}) - \aleph(\nu_c) \cdot (1-t^3/\rho_\nu)^{1/3}+o\big((1-t^3/\rho_\nu)^{1/3}\big)
& \text{ for $\nu = \nu_c$.}
\end{cases}
\]
\end{lemm}
\begin{proof}
All the computations are available in the companion Maple file \cite{MapleU}. From now on, we always consider $U$ as a power series in $t^3$. We first prove that its coefficients are nonnegative. Let us write $F(w) = t^3 \cdot tQ_1 (t)$ with $w=t^3$. It is the generating series of triangulations of the $1$-gon with weight $w$ per vertex and $\nu$ per monochromatic edge. As a power series in $w$, the series $F'$ has obviously nonnegative coefficients since it is the generating series of triangulations of the $1$-gon with a distinguished vertex. To prove that the coefficients of $U$ are nonnegative, we will prove that: 
\begin{equation}\label{eq:F'}
F'(w) = \frac{\nu +1}{2 \nu} U(t^3)
\end{equation}
This striking identity has a combinatorial interpretation, but discussing it here would take us too far from the subject of this article. We plan to return to it in a future work.

Let us denote by $U^\bullet$ the derivative of $U$ with respect to $w$. By differentiating \eqref{eq:t3U} with respect to $w$, we obtain
\[
1 = \frac{\mathrm d}{\mathrm d U} \left( U \frac{(1+\nu) U - 2}{32 \nu^3 (1-2U)^2} \, \phi (U) \right) \cdot U^\bullet,
\]
with
\[
\phi(U) = \Big( 8\nu (\nu+1)^2 U^3 -(11 \nu + 13)(\nu +1) U^2 +2 (\nu +3)(2\nu+1)U -4 \nu \Big).
\]
This gives us the following expression for $U^\bullet$ in terms of $U$:
\[
U^\bullet = \frac{8 \nu^3 (1-2U)^3}{\Big( 3(\nu+1)U^2 - 3 (\nu+1) U +\nu \Big) \Big( -4(\nu+1)^2 U^3 + 3 (\nu+1) (\nu+3) U^2 -6 (\nu +1) U +2 \Big)}.
\]
Furthermore, \eqref{eq:expQ1Q3} reads $F(w) = \dfrac{\widetilde{R_1}(U,\nu)}{\nu^4 (1-2U)^2}$.
Therefore we have
\[
F'(w) = \frac{\mathrm d}{\mathrm d U} \left( \frac{\widetilde{R_1}(U,\nu)}{\nu^4 (1-2U)^2} \right) \cdot U^\bullet.
\]
From there a tedious but basic computation yields \eqref{eq:F'}.

\bigskip

We now prove that the radius of convergence $w_0(\nu)$ of $U$ is equal to $\rho_\nu$. First, from \eqref{eq:F'}, we deduce that $w_0$ is also the radius of convergence of $F'$. Therefore, we can see that $w_0(\nu)$ is a non-increasing function of $\nu$. In addition, it is continuous. Indeed, for $\nu_1 \leq \nu_2$ and for every $w \geq 0$, we have that:
\begin{align*}
F'(\nu_1,w) \leq F'(\nu_2,w) &= \frac{1}{2}\sum_{M\in \bns{1}} |V(M)| w^{|V(M)| - 1} \nu_2^{m(M)}\\
&= \frac{1}{2}\sum_{M\in \bns{1}} |V(M)| w^{|V(M)| - 1} \nu_1^{m(M)} \, \left( \frac{\nu_2}{\nu_1}\right)^{m(M)}\\
&\leq \frac{1}{2}\sum_{M\in \bns{1}} |V(M)| w^{|V(M)|- 1} \nu_1^{m(M)} \, \left( \frac{\nu_2}{\nu_1}\right)^{|E(M)|}\\
&\leq \frac{1}{2}\sum_{M\in \bns{1}} |V(M)| w^{|V(M)|- 1}  \nu_1^{m(M)} \, \left( \frac{\nu_2}{\nu_1}\right)^{3|V(M)|-4} \leq F'(\nu_1,w\,(\nu_2/\nu_1)^3).
\end{align*}
Therefore, 
\[
\left(\frac{\nu_1}{\nu_2} \right)^3 w_0(\nu_1) \leq w_0(\nu_2) \leq w_0(\nu_1),
\]
which proves that $w_0(\nu)$ is continuous. Known results about triangulations without spins ensure that $w_0(1)>0$. Combined with the previous inequalities, it implies that $w_0(\nu)>0$ for every $\nu>0$. 

Since the series $U$ has nonnegative coefficients, its radius of convergence is a singularity by Pringsheim's Theorem. Therefore it is amongst the roots of the discriminant of the algebraic equation satisfied by $U$. This discriminant factorises into $P_1(\nu,w) \cdot P_2(\nu,w)$ given by \eqref{eq:P1} and \eqref{eq:P2} and we have to identify the correct root.

First, we start by identifying the values of $\nu$ for which $P_1$ and $P_2$ have a common positive root. The resultant of these two polynomials in $w$ factorises into several terms:
\begin{align*}
\mathrm{Resultant} (P_1,P_2) & = 4194304\nu^{12} \, (13573 \nu^4-54292\nu^3 + 69811 \nu^2-31038 \nu+67482) \\
& \qquad \cdot (6-14\nu+7\nu^2)^3 \, (\nu+1)^4 \, (\nu-3)^4.
\end{align*}
The factor of degree $4$ has no positive root and is irrelevant to us. For $\nu = 3$, it is easy to verify that the common root of $P_1$ and $P_2$ is negative. This leaves the factor of degree $2$ in $\nu$. Its roots are $\nu_c$ and $1-\frac{\sqrt 7}{7}$. Again, when $\nu = 1-\frac{\sqrt 7}{7}$, it is easy to verify that the common root of $P_1$ and $P_2$ is negative. This leaves $\nu_c$, for which we can verify that the common root of $P_1$ and $P_2$ is $\rho_{\nu_c}$, and since all the other roots of $P_1$ and $P_2$ are not positive real numbers, it implies that $w_0(\nu_c)=\rho_{\nu_c}$.

We now turn our attention to values of $\nu$ different from $\nu_c$. In those cases, we know that $\rho_\nu$ is a root of $P_1$ or of $P_2$, but cannot be a common root. It remains to identify the correct root.
The discriminant of $P_2$ is $-27648 \nu^2 \, (\nu+1)^3 \,(\nu-3)^3$. Thus, for $\nu > 3$, both roots of $P_2$ are imaginary. An easy analysis shows that for $\nu \in (0,3]$, only one of its two roots can take nonnegative values, it is the root given by
\[
w_2(\nu) = \frac{(\nu+1) (3-\nu) \sqrt{(\nu+1) (3-\nu)} - 9 (\nu-1) (\nu^2 -2 \nu -1)}{576 \nu^3},
\]
which is by Definition \ref{def:rho} equal to $\rho_\nu$ for every $\nu \in (0,\nu_c]$.

The discriminant of $P_1$ is $82556485632\nu^{18}(\nu-1)^2(\nu^2-2 \nu-7)^2(\nu+1)^3(\nu-3)^3$. Therefore, for $\nu < 3$, it has a unique (possibly with some multiplicity) real root and for $\nu >3$, it has three real roots (one is a double root for $\nu = 1+2\sqrt{2}$). The following situation is illustrated in Figure \ref{fig:rho}.
Among the three branches of $P_1$, the one that is real for every $\nu>0$ is decreasing for $\nu\geq \nu_c$. We denoted earlier this branch by $w_1(\nu)$.  By Definition \ref{def:rho}, it is equal to $\rho_\nu$ for every $\nu \geq \nu_c$. The other branches of $P_1$ are also real for $\nu \geq 3$, one stays negative and the other can be positive, is increasing in $\nu$ and intersects $w_1$ at $\nu = 1+2\sqrt{2}$ (its is called the second branch of $P_1$ in Figure \ref{fig:rho}).

From the previous description of the roots of $P_1$ and $P_2$, we have $w_0(3) = w_1(3)$ since it is the only positive root for this value of $\nu$. The fact that $w_0$ is nonincreasing in $\nu$ then implies that $w_0(\nu) = w_1(\nu) = \rho_\nu$ for every $\nu \geq \nu_c$. A simple check for $\nu =1$ shows that $w_0(1) = w_2(1)$ and, since $w_0$ is continuous and $w_1$ and $w_2$ are only equal at $\nu_c$, we have $w_0(\nu) = w_2(\nu) = \rho_\nu$ for $\nu \in (0,\nu_c]$.

\bigskip

We now turn to the claim that $U$ has a unique dominant singularity. We have to identify the roots of $P_1$ and $P_2$ other that the radius of convergence $\rho_\nu$ that are on the circle of convergence and test if they correspond to singularities. The following computations are done in the Maple companion file \cite{Maple}. At several occasions, we will need to know the value of $U(\rho_\nu)$. The algebraic equation \eqref{eq:t3U} for $U(t^3)$ writes:
\[
U = t^3 \cdot \psi (U).
\]
It is classical that the value of $U$ at its radius of convergence is the smallest positive root of the caracteristic equation
\[
\psi(U) - U \psi'(U)
\]
which is a rational fraction whose numerator has three factors given by
\begin{align}
\psi_1 (U) &= 2U-1, \notag\\
\psi_2 (U) &= 3(\nu+1)U^2 - 3 (\nu+1)U +\nu, \label{eq:Urho2}\\
\psi_3 (U) &= 4(\nu+1)^2 U^3 - 3 (\nu +1)(\nu+3) U^2 + 6 (\nu+1) U -2 \label{eq:Urho3}.
\end{align}
The three factors have common roots for $\nu \in \{ 1- \sqrt 7 /7 , 1, \nu_c , 3\}$. For $\nu = 3$, the smallest positive root is $1/8$, which is a root of $\psi_3$ and not of the other two factors. For $\nu = \nu_c$, only $\psi_2$ and $\psi_3$ have a common root, it is also the smallest positive root and is therefore $U(\rho_{\nu_c})$. For $\nu = 1$, $\psi_1$ and $\psi_3$ have $1/2$ as common root, but the smallest positive one is a root $\psi_2$. Finally, for $\nu =  1- \sqrt 7 /7$, $\psi_2$ and $\psi_3$ have a common root, but the smallest positive one is a root of $\psi_2$ alone. In conclusion, $U(\rho_\nu)$ is the smallest positive root of $\psi_2$ for $\nu \leq \nu_c$, and the smallest positive root of $\psi_3$ for $\nu \geq \nu_c$. Furthermore, we have $U(\rho_\nu) < 1/2$ for every $\nu >0$.

\medskip

We first look at the roots of $P_2$. For $\nu \leq \nu_c$ the radius of convergence is the positive root of $P_2$. For these values of $\nu$, $P_2$ has two real roots and it is easy to check for which values of $\nu$ these two roots are opposite one from another. It only happens when $\nu=1$ and we know that $U$ has no other dominant singularity than its radius of convergence since it corresponds to the derivative of the generating series of triangulations with a critical site percolation.

When $\nu \geq 3$, $P_2$ has two complex conjugate roots and $\rho_\nu$ is a root of $P_1$. We can compute the modulus of the roots of $P_2$ and see that it is increasing and larger that $\rho_\nu$ for $\nu = 3$, therefore, for $\nu \geq 3$ no roots of $P_2$ have the same modulus than $\rho_\nu$. 

It remains to check the roots of $P_2$ for $\nu \in (\nu_c,3)$. For these values of $\nu$, the roots of $P_2$ are real and differ from $\rho_\nu$. The only possibility for them to be on the circle of convergence is to be equal to $-\rho_\nu$. A Puiseux expansion of the solutions of \eqref{eq:t3U} with the explicit values of these two roots is possible and gives two possible branches. To identify which one of these two branches corresponds to $U$, we look at the constant term of their Puiseux expansion. Since $U$ has nonnegative coefficients, $|U(\rho_\nu)|\geq |U(z)|$ for any $z\in \mathbb{C}$ with $|z|=\rho_\nu$. Out of these two branches, one is singular but gives a value at the root larger than $U(\rho_\nu)$ that was computed previously. The other one is the branch corresponding to $U$ and is not singular. 

\medskip

We now turn to the roots of $P_1$. First, when $\nu \geq 3$, the polynomial $P_1$ has three real roots. By computing the resultant in $\rho$ of $P_1(\rho)$ and $P_1(-\rho)$ we identify the values of $\nu$ for which $P_1$ has two opposite roots. There are six values of $\nu$ for which it occurs. Three of these values are larger than $3$. Two of these values correspond to roots outside the circle of convergence (meaning that $\rho_\nu$ is the third root of $P_1$ for these values of $\nu$). The last possibility is $\nu = 1 + \frac{2}{9} \sqrt{136 - 10 \sqrt{10}}$ for which $\rho_\nu$ and $-\rho_\nu$ are both roots of $P_1$. We then check that for this specific value of $\nu$, the series $U$ is not singular at $-\rho_\nu$.

For $\nu \in [\nu_c,3]$, the roots of $P_1$ are $\rho_\nu$ and two complex conjugates. If one of the complex roots are on the circle of convergence, all three roots of $P_1$ have the same modulus. We can easily compute the cube of this modulus from the coefficients of $P_1$. It is then easy to check that this quantity is never equal to $\rho_\nu^3$, meaning that for this range of values for $\nu$, the three roots of $P_1$ never have the same modulus.

Finally, it remains to check the roots of $P_1$ for $\nu \in (0,\nu_c)$. Unfortunately, there are three values of $\nu$ for which some of the roots of $P_1$ have modulus $\rho_\nu$ (which we recall is a root of $P_2$ for these values of $\nu$). We will show that the roots of $P_1$ are never singularities of $U$ for $\nu < \nu_c$ with Newton's polygon method. We denote by $w_3(\nu)$ any root of $P_1$ for $\nu < \nu_c$. Exact expressions for $w_3$ are too complicated to directly compute a singular expansion of $U$ around $w_3$ with a computer. Instead, using polynomial eliminations, we compute a polynomial $\mathrm{Pol}(W,V)$ whose coefficients only depend on $\nu$ such that 
\[
\mathrm{Pol}(w_3-w,U(w_3)-U(w)) = 0
\]
for all $\nu < \nu_c$. This method is inspired by \cite[proof of Proposition 3.3]{BeCuMie}.

From \eqref{eq:t3U}, we can define a polynomial $\mathrm{alg_U(X,Y)} \in \mathbb Z[\nu] [X,Y]$ such that
\[
\mathrm{alg}_U(w,U(w)) = 0 \quad \text{for $|w| \leq \rho_\nu$}.
\]
We then define $A(Y)$ to be the resultant of $\mathrm{alg}_U(X,Y)$ and $P_1(X)$ with respect to $X$ so that, for every $\nu$, we have $A(U(w_3)) = 0$. The polynomial $A(Y)$ factorizes into to factors, one of them is the polynomial $\psi_3(Y)$ of \eqref{eq:Urho3} that gives the value of $U(\rho_\nu)$ when $\nu \geq \nu_c$ and the other, of degree $9$ in $Y$, will be denoted by $\tilde A (Y)$. We have to establish if $U(w_3)$ is a root of $\psi_3$ or $\tilde A$ in order to continue. We know that for $\nu < \nu_c$, $U(\rho_\nu)$ is the root of $\psi_2$ given by $\frac{1}{2} \left(1 - \sqrt{\frac{3-\nu}{3(1+\nu)}} \right)$. In addition, if $w_3$ is on the circle of convergence, we have $|U(w_3)| < U(\rho_\nu)$ by Pringsheim's theorem. Finally, we also saw that $\psi_3$ has a positive root for $\nu \leq \nu_c$, and that it is stricly larger that $U(\rho_\nu)$ for $\nu < \nu_c$. The discriminant of $\psi_3$ is negative for $\nu < 3$, so it has two imaginary conjugate roots. We can write an equation for the common squared modulus $|w|^2$ of these complex roots from the coefficients of $\psi_3$ by eliminating the real root and the sum of the two imaginary roots. This equation is given by
\[
- 16 (\nu +1)^4 |w|^6 +12 (\nu +1) ^2 (\nu +3) |w|^4 - 6 (\nu+1)(\nu+3) |w|^2 +4 = 0.
\]
The resultant of this polynomial (in $|w|$) with $\psi_2$ (that gives $U(\rho_\nu)$) vanishes only when $\nu = 3$. We can also check that the roots of $\psi_3$ for $\nu =1$ are all $1/2 > U(\rho_1)$. In conclusion, the roots of $\psi_3$ all have modulus stricly larger that $U(\rho_\nu)$ for $\nu < \nu_c$. This in turns shows that, for $\nu < \nu_c$, if $w_3$ is on the circle of convergence of $U$, then $U(w_3)$ is a root of $\tilde A$ and not a root of $\psi_3$.

We now define the following polynomial in $w,V$:
\[
B(w,V) = \mathrm{Resultant}(\tilde A(Y), \mathrm{alg}_U(w,Y-V) , Y)\]
so that $B(w, U(w_3) - U(w)) = 0$. Finally, we define
\[
\mathrm{Pol}(W,V) = \mathrm{Resultant}(B(w-W,V), P_1(w) , w)
\]
so that $\mathrm{Pol} (w_3 - w, U(w_3) - U(w)) =0$ for $\nu \leq \nu_c$ as announced. This polynomial factorizes into two factors, one of degree $18$ in $W$ and one of degree $9$. It is easy to check that only the factor of degree $9$ vanishes for $W=V=0$. We call $\widetilde{\mathrm{Pol}}$ this factor, so that $\widetilde{\mathrm{Pol}} (w_3 - w, U(w_3) - U(w)) =0$ for $\nu \leq \nu_c$.
We can then apply Newton's polygon method to $\widetilde{\mathrm{Pol}}$ to see that $U$ is not singular at $w_3$ for $\nu < \nu_c$.

\bigskip

To finish the proof, we have to establish the singular behaviour of $U$ at $\rho_\nu$. We can do so in a very similar fashion as above. 
%For $\nu \geq \nu_c$, we already computed the polynomial $A$ that satisfies $A(U(\rho_\nu)) = 0$ for $\nu \geq \nu_c$ since $\rho_\nu$ is a root of $P_1$ for these values of $\nu$. We also computed the polynomial $\psi$ that also satisfies $\psi(U(\rho_\nu))$\cmar{$=0$ ?}. The polynomials $A$ and $\psi$ both have $\psi_3$ as factor. The resultant of $\tilde A(U)$ and $(1-2U) \psi_2(U)$ with respect to $U$ only vanishes when $\nu = 3$, for which the common root is $U=1/2$ which we know is not $U(\rho_\nu)$. 
Recall that for $\nu \geq \nu_c$, $U(\rho_\nu)$ is a root of $\psi_3$. We define the polynomial
\[
B_3(w,V) = \mathrm{Resultant}(\psi_3(Y), \mathrm{alg}_U(w,Y-V) , Y)\]
so that $B_3(w, U(\rho_\nu) - U(w)) = 0$ for $\nu \geq \nu_c$. Finally, we define
\[
\mathrm{Pol}_3(W,V) = \mathrm{Resultant}(B_3(w-W,V), P_1(w) , w)
\]
so that $\mathrm{Pol}_3 (\rho_\nu - w, U(\rho_\nu) - U(w)) =0$ for $\nu \geq \nu_c$. This polynomial factorizes into two factors, one of degree $6$ in $W$ and one of degree $3$. It is easy to check that only the factor of degree $3$ vanishes for $W=V=0$. We call $\widetilde{\mathrm{Pol}}_3$ this factor, so that $\widetilde{\mathrm{Pol}}_3 (\rho_\nu - w, U(\rho_\nu) - U(w)) =0$ for $\nu \geq \nu_c$.
We can then apply Newton's polygon method to $\widetilde{\mathrm{Pol}}_3$ to see that $U$ has a typical square root singularity as announced at $\rho_\nu$ for $\nu \geq \nu_c$, except maybe at $\nu = \nu_c$ and $\nu = 3$ where the relevant coefficient of $\widetilde{\mathrm{Pol}}_3$ vanishes. The case $\nu = 3$ also gives a square root singularity and the case $\nu = \nu_c$ gives a singularity with exponent $1/3$ as required.

The situation for $\nu \leq \nu_c$ is very similar, %except that we first define $A_2(Y)$ to be the resultant of $\mathrm{alg}_U(X,Y)$ and $P_2(X)$ with respect to $X$ so that, for every $\nu \leq \nu_c$, we have $A_2(U(\rho_\nu)) = 0$. This polynomial has $\psi_2$ as factor, and we can check that
but with $\psi_2(U(\rho_\nu)) = 0$ for $\nu \leq \nu_c$ instead of $\psi_3(U(\rho_\nu)) = 0$. From here we can define $B_2$, $\mathrm{Pol}_2$ and $\widetilde{\mathrm{Pol}}_2$ is a similar fashion as $B_3$, $\mathrm{Pol}_3$ and $\widetilde{\mathrm{Pol}}_3$ so that $\widetilde{\mathrm{Pol}}_2 (\rho_\nu - w, U(\rho_\nu) - U(w)) =0$ for $\nu \leq \nu_c$ and Newton's polygon method gives the announced singular behaviour.
\end{proof}

\begin{prop}\cite[Claim 24]{BernardiBousquet}\label{prop:asyQ1}
The series $tQ_1$ and $t^3Q_3$ are Ising algebraic.
\end{prop}
\begin{proof}
First, the fact that $tQ_1$ and $t^3Q_3$ have a rational expression in $\nu$ and $U$ (recall \eqref{eq:expQ1Q3}) ensure that both series are algebraic by closure properties of algebraic functions since $U$ itself is algebraic.
From the proof of Lemma \ref{lem:etU}, we know that $|U(t^3)| < 1/2$ for all $|t| < t_\nu$. Therefore, the series $U$ and $(1-2U)^{-1}$ seen as series in $t^3$ have the same unique dominant singularity at $\rho_\nu$. The form of $tQ_1$ and $t^3Q_3$ given in \eqref{eq:expQ1Q3} then implies that these two series are algebraic and that their only singularities are also those of $U$.

The singular behaviour of $tQ_1$ is easily computable by integration (see \cite{FS}, Theorem VI.9 p. 420) from \eqref{eq:F'}, where we recall that $F(w) = t^3 \, tQ_1(t)$ with $w=t^3$. The singular behaviour of $t^3Q_3$ at $\rho_\nu$ could be obtained by plugging the explicit singular expansion of $U$ into expression \eqref{eq:expQ1Q3} and tracking cancellations. A more direct and satisfying proof without tedious computations can also be obtained with Lemma \ref{le:combi} of the following section, whose proof only relies on the Ising algebraicity of $tQ_1$ that we just established. Indeed, this lemma provides the singular behaviour of $t^3Q_3$ at $\rho_\nu$ and we know that this series has no other dominant singularity.
\end{proof}

\subsubsection{Asymptotic behaviour for triangulations by pinching}
Let $\bns{p,\kP}$ be the subset of triangulations with a boundary of
length $p$, whose boundary satisfies a property $\kP$ depending only on
the length, shape and spins of the boundary (in particular not on the
vertices, faces or edges in the interior regions).
Let $\gnstr[p]^{\kP}$ denote the generating function of triangulations in
$\bns{p,\kP}$.

\begin{lemm}\label{le:combi}
If $\gnstr[p]^\kP$ is algebraic then $\rho_\nu$ is a dominant singularity of $t^3\rightarrow t^pQ_p^\kP(\nu,t)$. In addition, for any $\nu>0$ and any positive integer $p$, there exist constants $\alpha_p(\nu)$, $\beta_p (\nu)$ and $\gamma_p(\nu)$ such that, $\gnstr[p]^\kP$ satisfies the following singular expansion at $\rho_c$:
\begin{itemize}
\item For $\nu\neq \nu_c$
\[
t^pQ_p^\kP(t)=\alpha_p^\kP(\nu)+\beta_p^\kP(\nu)(1-t^3/\rho_\nu)+\gamma_p^\kP(\nu)(1-t^3/\rho_\nu)^{3/2}+o((1-t^3/\rho_\nu)^{3/2}).
\]
\item For $\nu=\nu_c$
\[
t^pQ_p^\kP(t)=\alpha_p^\kP(\nu_c)+\beta_p^\kP(\nu_c)(1-t^3/\rho_{\nu_c})+\gamma_p^\kP(\nu_c)(1-t^3/\rho_{\nu_c})^{4/3}+o((1-t^3/\rho_{\nu_c})^{4/3}).
\]
\end{itemize}
\end{lemm}
\begin{rema}
We stress the fact that the Lemma does not state that $t^pQ_p^\kP$ is Ising-algebraic. Indeed,  the simple bounds used in the proof do not rule
out the possibility that $Q_p^\kP$ as a function of $t^3$ has other non
real dominant singularities and that these singularities induce an
oscillatory behaviour of $[t^{3n}]t^pQ^{p,\kP}$. This is clearly
illustrated with the case of $Q_{1}(t)$ when viewed as a function of
$t$ instead of $t^3$. Therefore, to establish the Ising-algebraicity of $t^pQ_p^\kP$ with the help of this Lemma, we need to establish that it also has a unique dominant singularity.
\end{rema}
\begin{proof}
We first observe that there exist positive constants $k_p^{\kP}$ and $\tilde k_p^{\kP}$  such that for all $n\geq p$:
\[
k_p^{\kP}[t^{3n}]t\gnstr[1]\leq [t^{3n}]t^p\gnstr[p]^{\kP}
\leq \tilde k_p^{\kP}[t^{3n}]t\gnstr[1].
\]
There is indeed an injection from $\bns[n]{p,\kP}$ into $\bns[n+p+2]{1}$: given an element of $\bns[n]{p,\kP}$, attach a triangle to each side of the
boundary, glue all these triangles together and add two edges to create an outer face of the appropriate degree. Conversely, given a boundary satisfying the property $\mathcal P$, we can first open the root edge and insert a quadrangulation of the one-gon inside the opened region at the corner corresponding to the root vertex. Then, we can add a vertex inside each face of the boundary and join it with an edge to each vertex of the corresponding face. See Figure~\ref{fig:Injection} for an illustration. This gives an injection from $\bns[n]{1}$ into $\bns[n+ 2 p + 1]{p,\kP}$.
\begin{figure}[t]
\begin{center}
\includegraphics[width=0.9\textwidth]{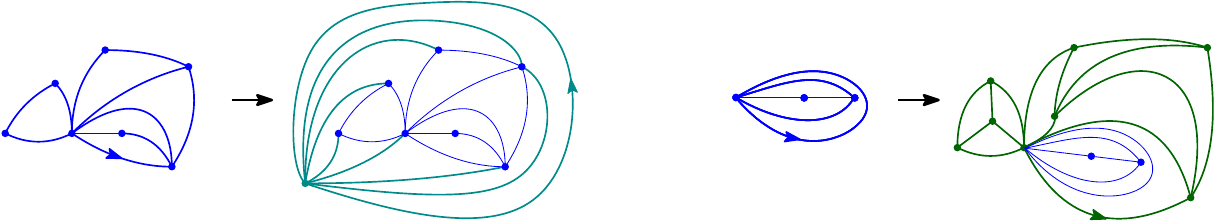}
\caption{\label{fig:Injection}Injection from $\bns[n]{p,\kP}$ into $\bns[n+p+2]{1}$ (left) and from $\bns[n]{1}$ into $\bns[n+2 p +1]{p,\kP}$ (right).}
\end{center}
\end{figure}

These bounds ensure that $\rho_\nu$ is the radius of convergence of $Q_p^{\kP}$ seen as a series in $t^3$ and Pringsheim's theorem then ensures that it is a singularity. The singular expansion follows from the classification of possible
singular behaviour of algebraic functions. Indeed, such functions have a Puiseux expansion in a slit neighbourhood of a singularity $\zeta$ of the form $f(z) = \sum_{k\geq k_0} c_k (z-\zeta)^{k/\kappa}$ with $k_0$ and $\kappa$ integers, see\cite[Thm VII.7 p.498]{FS}. Our bounds ensure that, at the singularity $\rho_\nu$, we have $\kappa = 2$ if $\nu \neq \nu_c$ and $\kappa = 3$ for $\nu = \nu_c$ and the expansion is of the form announced in the Lemma.
\end{proof}

For further use, notice that the case of the generating functions $Z_p$ and $Q_p$ of triangulations with a boundary (simple or not) of length $p$ is included in the statement of the Lemma.

\subsection{Triangulations with simple boundary}\label{sub:plus}

\subsubsection{Triangulations with positive boundary conditions}\label{sub:algebraicity}
We now state and prove our main technical result.
\begin{theo} \label{th:Aplus}
The series
\[
Z^+(y) := \sum_{p \geq 1} Z_{\ps^p}(\nu,t) y^p
\]
satisfies the following equation with one catalytic variable:
\begin{equation}
\label{eq:catA}
2 t^2 \nu (1-\nu) \left(\frac{Z^+(y)}{y} - Z_1^+\right) = y \cdot \mathrm{Pol} \left(\nu, \frac{Z^+(y)}{y}, Z^+_1,Z^+_2,t,y \right)
\end{equation}
where $Z^+_1 = [y] Z^+(y) = Z_{\ps}(\nu,t)$, $Z^+_2 = [y^2] Z^+(y) = Z_{\ps\ps}(\nu,t)$ and $\mathrm{Pol}$ is an explicit polynomial with integer coefficients, given in~\eqref{eq:polynom}.
\end{theo}
\begin{proof}
The difficulty of this result stems from the fact that we need \emph{two} catalytic variables to write a functional equation satisfied by $Z^+$. A technical application of Tutte's invariants method, introduced by Tutte (see \cite{Tutte}) and further developed in \cite{BernardiBousquet} allows us to derive the result. All the computations are available in the companion Maple file~\cite{Maple}.
\medskip

{\noindent \bf 1- A functional equation with two catalytic variables:}\\
The series $Z^+$ is a series with one catalytic variable. However, it is necessary to introduce a second catalytic variable to study it. Indeed, when writing Tutte-like equations by opening an edge of the boundary, a $\ns$ sign can appear on the newly explored vertex. It is then necessary to take into account triangulations with signs $\ns$ on the boundary. However, things are not hopeless as we can restrict ourselves to triangulations with a boundary consisting of a sequence of $\ps$ followed by a sequence of $\ns$.  Indeed, opening on the edge $\ps \ns$ of such a triangulation can only produce triangulations with the same type of boundary conditions. Figure \ref{fig:tutteAS} illustrates the different possibilities.
\begin{figure}[ht!]
\begin{center}
\includegraphics[width=0.8\textwidth]{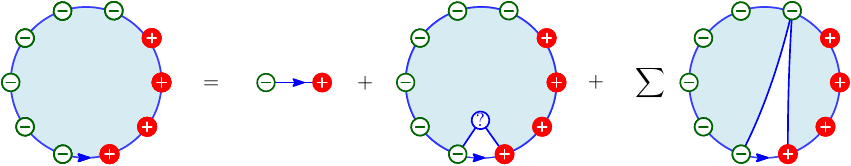}
\caption{\label{fig:tutteAS}Peeling along an interface.}
\end{center}
\end{figure}

Now, let us denote $Z^{+,-}(x,y)$ the generating series of triangulations with boundary conditions of the form $\ps^p\ns^q$ with $p+q \geq 1$, the variable $x$ being the variable for the number of $\ps$ and the variable $y$ being the variable for the number of $\ns$:
\[
Z^{+,-}(x,y) := \sum_{p+q \geq 1} \sum_{M\in \kT_f^{\ps^p\ns^q}}t^{|M|}\nu^{m(M)} \, x^p y^q.
\] 
Note that this series is symmetric in $x$ and $y$ and $Z^{+}(x) = Z^{+,-}(x,0)$. We also need the specialisation $Z^{+}_1 :=[x] Z^{+}(x)$ with positive boundary conditions and
the specialisation $Z^{+,-}_1 (x) := [y] Z^{+,-}(x,y)$ with a single $\ns$ spin (by symmetry we also have $Z^{+,-}_1(y) := [x] Z^{+,-}(x,y)$). The different possibilities illustrated in Figure~\ref{fig:tutteAS} translate into the following system of equations: 
\begin{equation}\label{eq:eqS}
\begin{cases}
Z^{+,-}(x,y) &= txy+\frac{t}{x}Z^{+,-}(x,y)Z^{+}(x)+\frac{t}{y}Z^{+,-}(x,y)Z^{+}(y) \\
 & \quad \quad +\frac{t}{x}(Z^{+,-}(x,y)-xZ^{+,-}_1(y))+\frac{t}{y}(Z^{+,-}-yZ^{+,-}_1(x)),\\
Z^{+}(x)&=\nu t x^2+\frac{\nu t}{x} (Z^{+}(x))^2 + \frac{\nu t}{x}(Z^{+}(x)-x Z^{+}_1)+\nu t Z^{+,-}_1(x).
\end{cases}
\end{equation}

{\noindent \bf 2- Kernel method:}\\
Following the classical kernel method, we write: 
\[
K(x,y) \cdot Z^{+,-}(x,y)=R(x,y)
\]
with
\[
K(x,y)=1-\frac{t}{x}-\frac{t}{y}-\frac{t}{x}Z^{+}(x)-\frac{t}{y}Z^{+}(y)
\]
and
\[
R(x,y) = txy-tZ^{+,-}_1(x)-tZ^{+,-}_1(y).
\] 
The next step is to find two distinct formal power series $Y_1(x),Y_2(x) \in \mathbb{Q}(x)[[t]]$ with coefficients in $\mathbb{Q}(x)$ which cancel the kernel, \emph{ie} such that $K(x,Y_i(x))=0$, for $i=1,2$. 

However, the equation $K(x,Y(x))=0$ can be written
\[
Y(x) = t \cdot \left( \frac{Y(x)}{x} + 1 + \frac{Y(x)}{x}Z^{+}(x) + Z^{+}(Y(x)) \right)
\]
and we can see that it has a unique solution in $\mathbb{Q}(x)[[t]]$ by computing inductively its coefficients in $t$. To get a second solution, we relax the hypothesis and ask for series $Y(x)\in \mathbb{Q}(x)[[t]]$ such that 
\[
K \left(x,\frac{Y(x)}{t}\right)=0.
\]
Note that this is possible because the series $t Z^{+}(y/t)$ is a well defined power series in $t$: indeed, the polynomial in $x$ given by $[t^n] (tZ^{+}(x))$ has degree at most $n - 1$, except for $n=2$ which has the term $\nu t^2 x^2$ for the map reduced to a single edge. Following advice given by Bernardi and Bousquet-Mélou in \cite{BernardiBousquet}, we perform the change of variables $x=t +at^{2}$ because of the term $t/x$ in the Kernel. The kernel equation now reads
\begin{equation}\label{eq:kernel}
Y \cdot (a-\nu Y) = t \cdot \left( 1 + at + Y \frac{Z^{+}(t+at^2)}{t^2} +\nu a Y^2 + (1+at) \left(Z^{+} (Y/t) - \nu Y^2 /t \right) \right)
\end{equation}
Using the fact that $Z^{+}(t+at^2) = \mathcal O(t^3)$ and $tZ^{+}(y/t) = \nu y^2 + \mathcal O(t)$, it is clear that if $Y$ is solution of this last equation, its constant term is either $0$ or $a/\nu$, and its coefficients in $t$ can then be computed inductively. This shows that the kernel equation has indeed two distinct solutions $Y_1$ and $Y_2$ such that
\begin{align*}
Y_1 (t) &= \mathcal O \left( t \right),\\
Y_2 (t) &= \frac{a}{\nu} + \mathcal O \left( t \right).
\end{align*}

{\noindent \bf 3 - Computation of invariants}\\
By writing
\[
K \left(x,\frac{Y_1(x)}{t}\right)=K \left(x,\frac{Y_2(x)}{t}\right)
\]
we can see that
\[
\frac{1}{Y_1} \left(  Z^{+}\left( \frac{Y_1}{t} \right) +1 \right) = \frac{1}{Y_2} \left(  Z^{+}\left( \frac{Y_2}{t} \right) +1 \right).
\]
Following Tutte we say that the quantity
\begin{equation} \label{eq:inv1}
I(y)=\frac{1}{y} \left(  Z^{+} \left( \frac{y}{t} \right) +1 \right)
\end{equation}
is an invariant since it takes the same value for $Y_1$ and $Y_2$ and we note this common value $\mathcal I$. To find a second invariant we have to dig deeper and look at all the other equations.

First, notice that since $K(x,Y_i/t) = 0$, we have two additional equations $R(x,Y_i/t) = 0$ that link $x,Y_i, Z^{+,-}_1(x)$ and $Z^{+,-}_1(Y_i/t)$. Solving this new system for $x$ and $Z^{+,-}_1(x)$ gives:
\begin{align*}
x &= t \cdot \frac{Z^{+,-}_1(Y_1/t) - Z^{+,-}_1(Y_2/t)}{Y_1 - Y_2},\\
Z^{+,-}_1(x) &= \frac{Y_2 \, Z^{+,-}_1(Y_1/t) - Y_1 \, Z^{+,-}_1(Y_2/t)}{Y_1 - Y_2}.
\end{align*}
Using equation \eqref{eq:eqS} linking $Z^{+}(y)$, $Z^{+,-}_1(y)$ and $y$ together with the expression of $\mathcal I$, we can express $Z^{+,-}_1(Y_i/t)$ in terms of $Y_1$, $Y_2$ and $\mathcal I$:
\[
Z^{+,-}_1(Y_i/t) = -t Y_i \, \mathcal{I}^2 + \frac{Y_i + \nu t^2}{\nu t} \,  \mathcal{I} + \frac{\nu Z^{+}_1 t^2 - t - \nu Y_i^2}{\nu t^2}
\]
for $i=1,2$. We can now use either of these last two identities to express $x$ and $Z^{+,-}_1(x)$ in terms of $Y_i$ and $\mathcal I$:
\begin{align}
\label{eq:xI}
x &= -t^2 \, \mathcal I^2 +\frac{\mathcal I}{\nu} - \frac{Y_1 + Y_2}{t},\\
\label{eq:SxI}
Z^{+,-}_1(x) &= -t \mathcal I - \frac{Y_1 Y_2}{t^2} - \frac{1}{\nu t} +Z^{+}_1.
\end{align}
Using once again the kernel equation, we can express $Z^{+}(x)$ in terms of $x$, $Y_1$ and $Z^{+}(Y_1/t)$ and therefore solely in terms of $Y_1$, $Y_2$ and $\mathcal I$:
\begin{equation}
\label{eq:AxI}
Z^+(x) = t^3 \, \mathcal I ^3 - \frac{\nu +1}{\nu} \,t \, \mathcal I^2 + \left(\frac{1}{\nu t} +Y_1+Y_2 \right) \mathcal I - \frac{Y_1+Y_2}{t^2} - 1.
\end{equation}
Finally, putting our expressions \eqref{eq:xI}, \eqref{eq:SxI} and \eqref{eq:AxI} into the second equation of \eqref{eq:eqS} verified by $Z^{+}(x)$ gives an equation linking $Y_1$, $Y_2$ and $\mathcal I$:
\begin{equation}
\label{eq:Inv2.0}
Y_1^2 + Y_2^2 + Y_1 Y_2 + \beta(\mathcal{I}) \, (Y_1 + Y_2) + \delta(\mathcal{I}) = 0
\end{equation}
where
\begin{align*}
\beta(X) &= \frac{\nu t^4 \, X^2 + (2 \nu -3) t^2 \, X+ 1 - \nu}{\nu t},\\
\delta(X) &= \frac{2 t^4 \nu(\nu - 1) \, X^3 + t^2 (2 - \nu ( \nu + 1) ) \, X^2 + (\nu -1) X + 2 \nu t - \nu^2 t - 2 \nu^2 t^2 \, Z^{+}_1}{\nu^2}.
\end{align*}
At this point, we almost have a second invariant and just have to isolate $Y_1$ and $Y_2$ in \eqref{eq:Inv2.0} to get it. Following the guidelines of \cite{BernardiBousquet}, we want to perform a change of variables to transform equation \eqref{eq:Inv2.0} into an equation of the form
\begin{equation}
\label{eq:Ui}
U_1^2 + U_2^2 + U_1 U_2 - 1 = 0
\end{equation}
where $U_i=U(Y_i)$ for some function $U$.
Since
\[
(U_1 - U_2)(U_1^2 + U_2^2 + U_1 U_2 - 1) = (U_1^3 - U_1) - (U_2^3 - U_2)
\]
we will deduce that
\[
U_1^3 - U_1 = U_2^3 - U_2
\]
giving us our second invariant.

First, setting $Y_i = X_i - \frac{1}{3} \beta(\mathcal{I})$ yields the following:
\[
X_1^2 + X_2^2 + X_1 X_2 + \delta(\mathcal{I}) - \frac{\beta(\mathcal{I})^2}{3}= 0.
\]
Now, setting
\[
U_i = \frac{X_i}{\sqrt{\frac{\beta(\mathcal{I})^2}{3} - \delta(\mathcal{I})}} = \frac{Y_i + \frac{1}{3} \beta(\mathcal{I})}{\sqrt{\frac{\beta(\mathcal{I})^2}{3} - \delta(\mathcal{I})}}
\]
We have equation \eqref{eq:Ui} and thus
\[
\frac{Y_1 + \frac{1}{3} \beta(\mathcal{I})}{\sqrt{\frac{\beta(\mathcal{I})^2}{3} - \delta(\mathcal{I})}}\cdot \left( \left(\frac{Y_1 + \frac{1}{3} \beta(\mathcal{I})}{\sqrt{\frac{\beta(\mathcal{I})^2}{3} - \delta(\mathcal{I})}}\right)^2 -1\right) = \frac{Y_2 + \frac{1}{3} \beta(\mathcal{I})}{\sqrt{\frac{\beta(\mathcal{I})^2}{3} - \delta(\mathcal{I})}}\cdot \left( \left(\frac{Y_2 + \frac{1}{3} \beta(\mathcal{I})}{\sqrt{\frac{\beta(\mathcal{I})^2}{3} - \delta(\mathcal{I})}}\right)^2 -1\right).
\]
Now we just have to transform the last equation into an equality with no radicals
\begin{align*}
\left(Y_1 + \frac{1}{3} \beta(\mathcal{I}) \right) &\cdot \left( \left(Y_1 + \frac{1}{3} \beta(\mathcal{I}) \right)^2 - \left(\frac{\beta(\mathcal{I})^2}{3} - \delta(\mathcal{I}) \right) \right) \\
&= \left(Y_2 + \frac{1}{3} \beta(\mathcal{I}) \right)\cdot \left( \left(Y_2 + \frac{1}{3} \beta(\mathcal{I}) \right)^2 - \left(\frac{\beta(\mathcal{I})^2}{3} - \delta(\mathcal{I}) \right) \right)
\end{align*}
to get our second invariant
\begin{equation*}
\label{eq:Inv2.1}
J(y) = \left(y + \frac{1}{3} \beta(I(y)) \right)\cdot \left( \left(y + \frac{1}{3} \beta(I(y)) \right)^2 - \left(\frac{\beta(I(y))^2}{3} - \delta(I(y)) \right) \right).
\end{equation*}

Of course, if we eliminate from $J(y)$ the terms depending on $y$ only through $I(y)$, we still get an invariant. It is given by
\begin{align}
\label{eq:Inv2.2}
\widetilde{J}(y) & = 2 \frac{\nu - 1}{\nu} \, t^4 \, y \, I(y)^3 + \left( t^3 \, y^2  - \frac{\nu^2+\nu-2}{\nu^2} \, t^2 \, y \right) I(y)^2 + \left( \frac{2 \nu - 3}{\nu} \, t \, y^2 + \frac{\nu -1}{\nu^2} \, y \right) I(y) \\
& \quad \quad + y^3 - \frac{\nu-1}{\nu \, t} \, y^2 + \left(\frac{\nu-2}{\nu} + 2 \,t \, Z^{+}_1(t) \right) \, t \, y. \notag
\end{align}
The two invariants $J$ and $\widetilde J$ contain the same information and we will work with $\widetilde J$ to shorten computations.

\bigskip

{\noindent \bf 4 - $\tilde J$ is a polynomial function of $I$}\\
Borrowing again from Tutte and Bernardi-Bousquet M\'elou \cite{BernardiBousquet}, we now show that $\widetilde J (y)$ is a polynomial in $I(y)$ with explicit coefficients depending only on $\nu$ and $t$. To that aim, we first notice that from expression \eqref{eq:inv1} we can easily write
\[
I(y) = \frac{1}{y} + R(y)
\]
where $R(y)$ is a series having no pole at $y=0$. Hence, from the form of $\tilde J$, we can find Laurent series $C_0(t), \, C_1(t)$ and $C_2(t)$ (depending on $\nu$) such that the series
\begin{equation}\label{eq:InvEq}
H(y) := \widetilde{J}(y) - C_2(t) \, I(y)^2 - C_1(t) \, I(y) - C_0(t)
\end{equation}
has coefficients in $t$ which are rational in $y$ and vanish at $y=0$. The computations of these coefficients is straightforward: we first eliminate the term in $1/y^2$ of $\widetilde J (y)$ , then the term in $1/y$ of $\widetilde{J}(y) - C_2(t) \, I(y)^2$ and finally the constant term of $\widetilde{J}(y) - C_2(t) \, I(y)^2 - C_1(t) \, I(y)$. The explicit values of the $C_i$'s are:
\begin{align*}
C_2(t) & = 2 \frac{\nu - 1}{\nu} \, t^4,\\
C_1(t) & = \frac{\nu - 1}{ \nu^2} \, t^2 \, \left(2 \nu t Z^{+}_1(t) - \nu -2  \right),\\
C_0(t) & = 2 \frac{\nu - 1}{\nu} \, t^2 \, Z^{+}_2(t) +  2 \frac{\nu - 1}{\nu} \, t^2 \, \left( Z^{+}_1(t) \right)^2 - \frac{\nu(\nu+1) - 2}{\nu^2} Z^{+}_1(t) +\frac{\nu^2 \, t^3 + \nu -1}{\nu^2}.
\end{align*}

We see from the expressions of the $C_i$'s and of $I(y)$ and $\widetilde{J}(y)$ that $H(y)$ is in fact a power series in $t$ with coefficients that vanish at $y=0$. Supposing that $H(y)$ is not $0$, we can write
\[
H(y) = \sum_{n \geq n_0} h_n(y) t^n
\]
and evaluating $H(Y_i)$ gives
\[
\left[ t^{n_0} \right] H(Y_i) = h_{n_0} \left(\left[ t^0 \right] Y_i \right).
\]
On the one hand we have $\left[ t^0 \right] Y_1 = 0$ so that $\left[ t^{n_0} \right] H(Y_1) = 0$. On the other hand, $\left[ t^0 \right] Y_2 = \frac{a}{\nu}$ and, $h_{n_0}(y)$ is different from $0$ by assumption and does not depend on $a$ since $H$ itself does not. Therefore we have $H(Y_1) \neq H(Y_2)$ which contradicts the fact that $H(y)$ is an invariant. This means that $H(y) = 0$. 
\bigskip

{\bf \noindent 5 - And finally an equation with only one catalytic variable!}\\
Now, replacing in the invariant equation \eqref{eq:InvEq} each $C_i$ by its value gives an equation for satisfied by $I(y)$:
\begin{align*}
0 = & 2 \nu (\nu -1) \, t^5 \, y \, I^3(y)  + \left( \nu^2 \, t \, y^2 - (\nu (\nu +1) -2)  \, y - 2 \nu ( \nu -1) t^2 \right) \,t^3 \, I^2(y) \\
& - \left(2 \nu (\nu -1 ) \, Z^{+}_1(t) + (\nu (3-2\nu)) \, t \, y^2 - (\nu -1) \, y + (2 - \nu (\nu+1)) \, t^2 \right) \,t \, I(y)\\
& + 2 \nu (\nu -1) \, t^3 \, (Z^{+}_1 (t) )^2 - \left( 2 \nu^2 t y + (2 - \nu (\nu +1 ) \right) t^2 \, Z^{+}_1(t) - 2 \nu (\nu -1) t^3 Z^{+}_2(t) \\
& + \nu^2 t y^3 - \nu (\nu -1) y^2 - \nu (\nu -2) t^2 y - \nu^2 t^4 - \nu t + t.
\end{align*}
Replacing $I(y)$ by its expression \eqref{eq:inv1} and performing the change of variable $y \to t y$ finally yields an equation with one catalytic variable for $M$:
\begin{equation} \label{eq:eqZ+}
\begin{split}
0 = & 2 \nu (\nu-1) \,t^2 \, (Z^{+}(y))^3 
+ ( \nu^2 t^3 y^2 - (\nu ( \nu -1) -2) t y + 4 t^2 \nu(\nu  -1)) (Z^{+}(y))^2 \\
& +\left(- 2 t^2 \nu ( \nu - 1)\, y \, Z^{+}_1(t) 
+ \nu (2 \nu -3) t^2 \, y^3 
+(2 \nu^2 t^3+\nu-1) y^2
- (\nu ( \nu -1) -2) t y
+2 t^2 \nu ( \nu - 1) \right) \,Z^{+}(y)\\
& - 2 t^2 \nu ( \nu - 1)\, y^2 \, (Z^{+}_1(t))^2
+(-2 y^3 t^3 \nu^2+ (\nu (\nu+1)-2)t y^2 - 2 \nu (\nu -1 ) t^2 y) \, Z^{+}_1(t)
- 2 \nu (\nu -1 ) t^2 y^2 \, Z^{+}_2(t)\\
& +y^5 t^3 \nu^2 -  \nu (\nu -1 ) t y^4+ 2 \nu (\nu - 1) t^2 y^3.
\end{split}
\end{equation}
This equation reads
\[
2 t^2 \nu (1-\nu) \left(\frac{Z^{+}(y)}{y} - Z^{+}_1\right) = y \cdot \mathrm{Pol} \left(\nu, \frac{Z^{+}(y)}{y}, Z^{+}_1,Z^{+}_2,t,y \right)
\]
with
\begin{equation}\label{eq:polynom}
\begin{split}
\mathrm{Pol} (\nu,a,a_1,a_2, t,y) = \, & \nu (\nu -1) t^2y + \nu (1- \nu) t y^2 + \nu^2 t^3 y^3 + 2 \nu ( 1-\nu) t^2 a_2 \\
& -(2 \nu ^2 t^2 y + 2 - \nu (1+\nu)) t a_1 + 2 \nu (1-\nu) t^2 a_1^2 + 2 \nu (1-\nu) t^2 a_1 a\\
& + \left((2-\nu(1+\nu))t + y (\nu-1) +2 \nu^2 t^3 y +\nu (2 \nu -3) t^2 y^2 \right)a\\
& + \left( \nu^2 t^2 y^2 + (2-\nu ( \nu+1)) y + 4 \nu (1-\nu ) t\right)t a^2 + 2 \nu (\nu-1) t^2 y a^3.
\end{split}
\end{equation}
\end{proof}

\subsubsection{Algebraicity and singularity of triangulations with positive boundary}
The equation with one catalytic variable \eqref{eq:catA} could be solved by the general methods of Bousquet-M\'elou and Jehanne \cite{BMJ} or even by guess and check \emph{\`a la} Tutte. We can also rely on the expressions for $Q_1$ and $Q_3$, obtained in~\cite{BernardiBousquet} and recalled in~Theorem~\ref{th:gsBoundary}, to solve it and even obtain a rational parametrization of $Z^+$, see \cite{Maple}. However, we will only need the following result:
\begin{prop}\label{prop:expRatM}
Recall from Theorem~\ref{th:gsBoundary}, that $U\equiv U(\nu,t)$ is the unique power series in $t$ having constant term $0$ and satisfying \eqref{eq:t3U}.
Each series $Z_{\oplus^p}=Z^{+}_p$ is
algebraic over $\mathbb{Q}(\nu,t)$. More precisely there exist polynomials $R_p^\nu$ in $U$ whose coefficients are rational in $\nu$, such that, for all $p\geq 1$:
\[
	t^{3p}\cdot t^pZ^{+}_p = \frac{R_p^\nu(U)}{(1-2U)^{3p}}.
\]
\end{prop}

\begin{proof}
We proceed by induction on $p$, the result is clear for $Z^{+}_1$ (which is equal to $Q_1$) by Theorem~\ref{th:gsBoundary}. For $p=2$, on the one hand, we can write a Tutte-like equation for triangulations contributing to $Z_{\ps\ps}$ and to $Z_{\ps\ns}$. Peeling the root edge of those triangulations yields the following equality:
\[
Z_{\ps\ps}(t)+\nu Z_{\ps\ns}(t) = 2t\nu +t \nu \left( Z_{\ps\ps}(t)+Z_{\ps\ns}(t) \right) + t \nu \left( Z_{\ps\ps\ps} (t) + 3 Z_{\ps\ps\ns}(t) \right).
\]
Since $Q_1(t)Q_2(t)$ enumerates the triangulations with a boundary of length 3 rooted on a loop with spin $\ps$, we also have 
\begin{equation}\label{eq:Z3moins}
Z_{\ps\ps\ps} (t) + 3 Z_{\ps\ps\ns}(t) = Q_3(t) - Q_1(t)Q_2(t).
\end{equation}
On the other hand, $Z_{\ps\ps}(t)+Z_{\ps\ns}(t)$ enumerates the triangulations with a simple boundary of length 2 and hence: 
\begin{equation}\label{eq:Z2moins}
	Z_{\ps\ps}(t)+Z_{\ps\ns}(t) = Q_2(t)-Q_1(t)^2.
\end{equation}
Combining these two relations and $Q_1(t) = \nu t Q_2(t)$, we obtain the following expression for $Z^+_2 = Z_{\ps\ps}$
\[
Z^+_2(t) = \frac{2 \nu}{1 - \nu} \, t + \frac{\nu}{1 - \nu} \, t Q_3 (t) - \frac{1}{(1 - \nu) t} \, Q_1 (t) - Q^2_1(t)
\]
which, with the expressions of $t Q_1$ and $t^3 Q_3$ given in Theorem~\ref{th:gsBoundary}, implies the statement for $t^6\cdot t^2Z^+_2$.
\smallskip

We now carry out an induction. By the result of Theorem~\ref{th:Aplus}, and more precisely by setting $y\leftarrow ty$ in equation~\eqref{eq:eqZ+} and then dividing it by $t^2$, we get:
\begin{equation*}
\begin{split}
2\nu(1-\nu)Z^{+}(ty) = &
2\nu(1-\nu) \, tZ^{+}_1 \, y
+ (1-\nu) \Big(2 \nu \, (tZ_1^+)^2 + 2 \nu \, t^2 Z_2^+ - ( \nu +2) \,tZ_1^+ \Big) y^2 \\
& - \nu t^3 (2 \nu \,tZ_1^+ - \nu + 1) y^3 - t^3 \nu (\nu-1) y^4 + \nu^2 t^6 y^5 \\
& + \Big( \nu t^3 (2\nu-3) y^3+(\nu+2 \nu^2 t^3-1) y^2-(\nu-1)(\nu+ 2 \nu \, tA1 +2) y \Big) Z^+(ty)\\
& + \Big( \nu^2 t^3  y^2 - (\nu+2)(\nu-1) y+ 4 \nu(\nu-1) \Big) (Z^+(ty))^2\\
& + 2\nu(\nu-1)(Z^{+}(ty))^3.
\end{split}
\end{equation*}
For $p \geq 3$, identifying the coefficients of $y^p$ on both sides leads to
\begin{equation*}
\begin{split}
2\nu(1-\nu) t^p Z_p^+ =
& - \nu t^3 (2 \nu \,tZ_1^+ - \nu + 1) \delta_{p=3} - t^3 \nu (\nu-1) \delta_{p=4} + \nu^2 t^6 \delta_{p=5} \\
& + \nu t^3 (2\nu-3) t^{p-3} Z_{p-3}^+ +(\nu+2 \nu^2 t^3-1) t^{p-2} Z_{p-2}^+ -(\nu-1)(\nu+ 2 \nu \, tA1 +2) t^{p-1} Z_{p-1}^+ \\
& +  \nu^2 t^3  [y^{p-2}] (Z^+(ty))^2 - (\nu+2)(\nu-1) [y^{p-1}] (Z^+(ty))^2+ 4 \nu(\nu-1)[y^{p}] (Z^+(ty))^2 \\
& + 2\nu(\nu-1) [y^p] (Z^{+}(ty))^3.
\end{split}
\end{equation*}
Multipliying this identity by $t^{3p} (1-2U)^{3p}$ gives the following recursion relation for the polynomials $R_p^\nu (U)$:
\begin{equation*}
\begin{split}
2\nu(1-\nu) R_p^\nu (U) =
& - \nu t^{3p}(1-2U)^{3p-2} (\widetilde{R_1} (U) /(16 \nu^3) - (\nu - 1) t^3 (1-2U)^2) \delta_{p=3} \\
& - t^{3(p+1)} (1-2U)^{3p} \nu (\nu-1) \delta_{p=4} + \nu^2 t^{3(p+2)} (1-2U)^{3p} \delta_{p=5} \\
& + \nu t^{12} (1-2U)^{9} (2\nu-3) R_{p-3}^\nu(U) \\
& +(\nu+ 2 \nu^2 t^3-1) t^6 (1-2U)^6 R_{p-2}^\nu(U) \\
& -(\nu-1)((\nu+2)t^3 (1-2U)^3 + 2 \,(1-2U) \widetilde{R}_1(U)/\nu^3 ) R_{p-1}^\nu(U) \\
& +  \nu^2 t^9 (1-2U)^6 \sum_{i+j = p-2} R_i^\nu (U) R_j^\nu (U)\\
&- (\nu+2)(\nu-1) t^3 (1-2U)^3 \sum_{i+j = p-1} R_i^\nu (U) R_j^\nu (U) \\
&+ 4 \nu(\nu-1)\sum_{i+j = p} R_i^\nu (U) R_j^\nu (U) \\
& + 2\nu(\nu-1) \sum_{i+j+k = p} R_i^\nu (U) R_j^\nu (U) R_k^\nu(U).
\end{split}
\end{equation*}
Using the fact that $32 \nu^3 t^3 (1-2U)^2$ is a polynomial in $U$ and $\nu$ with integer coefficients from \eqref{eq:t3U} finishes the proof.
\end{proof}

With these expressions of the series $Z_p^+$ in terms of $U$, following the exact same chain of arguments as in the beginning of the proof of Proposition~\ref{prop:asyQ1}, we obtain that these series are algebraic and that their singularities are singularities of $U$.  Combined with Lemma~\ref{le:combi}, it yields the following crucial result:
\begin{coro}\label{co:ZplusSingular}
For any $p$, the series $Z^+_p$ is Ising-algebraic.
\end{coro}

\subsubsection{Triangulations of the $p-$gon with arbitrary fixed boundary condition}\label{sub:anycondition}
Our starting point is the standard root-edge deletion equation for
triangulations of a $p$-gon with a given boundary word.
\begin{prop}\label{prop:PeelSimple}
Let $\omega = \omega_1 \cdots  \omega_k$ be a non-empty word on $\{\ps,\ns\}$ and let $a,b$ be in $\{\ps,\ns\}$, then we have:
\begin{equation}\label{eq:Zomega}
Z_{ab\omega} = \frac{\nu^{\delta_{a=b}}t}{1-2\nu^{\delta_{a=b}}(tZ_{\ps})}\left(Z_{a\ps b\omega}+Z_{a\ns b\omega}
+
\sum_{i=1}^k Z_{\omega_i b \omega_1 \cdots \omega_{i-1}}Z_{a\omega_i \omega_{i+1} \cdots \omega_k}\right).
\end{equation}
\end{prop}
\begin{figure}[t]
\begin{center}
\includegraphics[width=0.9\textwidth]{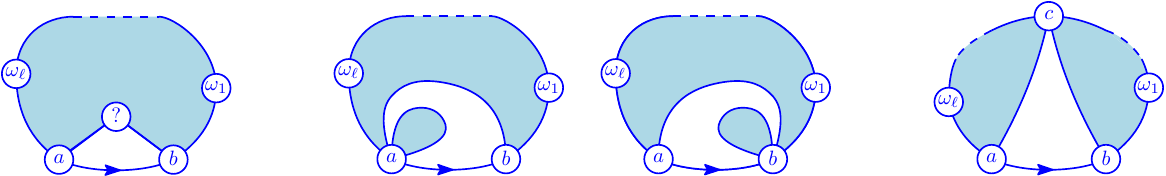}
\caption{\label{fig:RootDeletion}The 4 possible configurations for the inner face incident to the root-edge.}
\end{center}
\end{figure}
\begin{proof}
Let $T$ be an element of $\mathcal{T}_f^{b\omega a}$. Figure~\ref{fig:RootDeletion} illustrates the 4 possibilities for the configuration of the inner face incident to the root-edge of $T$ (i.e. the edge between the spins $a$ and $b$ by our rooting convention). The deletion of the root edge of $T$ translates into the following equation for the corresponding generating series:
\begin{equation}\label{eq:Zomeganum}
Z_{ab\omega}=\nu^{\delta_{a=b}}t\left(Z_{a\ps b\omega}+Z_{a\ns b\omega}
+Z_{ab\omega}\cdot(Z_{a}+Z_{b})+
\sum_{i=1}^k Z_{\omega_i b \omega_1 \cdots \omega_{i-1}}Z_{a\omega_i \omega_{i+1} \cdots \omega_k}\right),
\end{equation}
which yields~\eqref{eq:Zomega}.
\end{proof}

We can finally prove our main algebraicity result:
\begin{proof}[Proof of Theorem~\ref{theo:combimain}]
Fix $\nu >0$. The series $Z_\omega$ have a cyclic symmetry with respect to $\omega$. Indeed, if $\omega'$ is a cyclic permutation of $\omega$, we have clearly $Z_\omega = Z_{\omega'}$. In addidtion, if $\overline{\omega}$ denotes the image of $\omega$ under the involution that exchanges $\ominus$ and $\oplus$, we also have $Z_\omega = Z_{\overline{\omega}}$.
For every nonempty word $\omega$ on $\{\ps,\ns\}$, we denote by $|\omega|_\ominus$ the number of times $\ominus$ appears in $\omega$.
%
%\cmar{*On n'a pas besoin de la paramétrisation rationnelle, et on peut remplacer*}
%
%From (\ref{eq:Zomeganum}) and Proposition~\ref{prop:expRatM}, we
%see by induction on $(|\omega|,|\omega|_\ominus)$ in lexicographic order that for every non empty word $\omega$, there exists a polynomial $R^\nu_{\omega}$ in U whose coefficients are rational in $\nu$ such that
%\[
%t^{4|\omega|} Z_{\omega} = \frac{R^\nu_{\omega}(U)}{(1-2U)^{3|\omega|}}.
%\]
%Indeed, Proposition~\ref{prop:expRatM} ensures that this statement is true when $|\omega|_\ominus = 0$ while Equations \eqref{eq:Z3moins} and \eqref{eq:Z2moins} and the invariance of $Z_\omega$ under cyclic permutations or spin inversion ensure that the statement is true if $|\omega| \leq 3$. Equation \eqref{eq:Zomeganum} and invariance under cyclic permutation then allows to carry out the induction for $|\omega|>3$. From there, as for Corollary~\ref{co:ZplusSingular}, the same arguments as in the proof of Proposition~\ref{prop:asyQ1} establish the Ising algebraicity of $Z_{\omega}$ for every fixed $\omega$.
%
%\cmar{*par ça (si on a envie)*}
%
%\cmar{
We proceed by induction on $(|\omega|,|\omega|_\ominus)$ in lexicographic order. By Corollary~\ref{co:ZplusSingular}, we know that the result holds when $|\omega|_\ominus = 0$.  Equations \eqref{eq:Z3moins} and \eqref{eq:Z2moins} and the invariance of $Z_\omega$ under cyclic permutations or spin inversion ensure that the statement is true if $|\omega| \leq 3$.
From \eqref{eq:Zomeganum}, we see that for every non empty word $\omega$ and any $a,b\in\{\ps,\ns\}$, we can express $Z_{a\ns b \omega}$ as a linear combination of products of some $Z_{\omega'}$, where either $|\omega'|<| a\ns b \omega|$ or $|\omega'|_{\ns}<|a\ns b \omega|_\ns$ (or both). It implies that the singularities of $Z_{a\ns b \omega}$ are included in the set of singularities of those $Z_{\omega'}$. By the induction hypothesis and Lemma~\ref{le:combi}, it concludes the proof of the the Ising algebraicity of $Z_{\omega}$ for every fixed $\omega$.
%}
The asymptotic behavior of the coefficients of the series then follow directly by applying Proposition~\ref{prop:dvIsing}.
\end{proof}

\section{Local weak convergence of large triangulations}

\subsection{Local weak topology and tightness of the root degree}

To prove Theorem \ref{th:localCV}, we will first prove that the sequence of probability measures $\{ \mathbb P_n^\nu \}$ is tight for the topology of local convergence. 
Fix $(l_r)_{r\geq 1}$ a sequence of positive real numbers, denote $K_{(l_r)}$ the subset of $\mathcal T$ defined by: 
\[
	T\in K_{(l_r)} \text{ if and only if } \forall r\geq 1,\ \max_{v\in B_{r}(T)} \deg(v) \leq l_r.
\]
Then, $K_{l_r}$ is a compact subset of $(\mathcal{T},d_{loc})$.
To prove tightness, we will therefore prove that for every $r$ the maximum degree $L_r$ in a ball of radius $r$ for a random triangulation with law $\mathbb P_n^\nu$ is tight with respect to $n$. First, let us do so for the root vertex degree.

\begin{lemm} \label{lem:root-tight}
Let $X_n$ be the degree of the root vertex under $\mathbb P_n^\nu$. The sequence of random variables $(X_n)_{n\geq 1}$ is tight.
\end{lemm}
\begin{rema}
In Section \ref{sec:rootDegree}, we will prove that the limiting distribution of the $X_n$'s has exponential tails for $\nu=\nu_c$ (and the proof works in fact for $\nu$ close enough to $\nu_c$, see the remark following Proposition~\ref{prop:BoltzDegTail}). It may be possible to extend this statement to every $X_n$, with a uniform exponential upper bound for the tails. This is usually the approach to prove tightness results in random maps (see for example Lemmas 4.1 and 4.2 in \cite{AngelSchramm}).

It turns out that things become fairly technical in our setting and we are still unable to prove exponential tails for every value of $\nu$. However, though a much weaker statement, Lemma \ref{lem:root-tight} is sufficient to prove tightness and moreover has a very simple and robust proof that we were not able to find in the literature.
\end{rema}
\begin{proof}[Proof of Lemma \ref{lem:root-tight}]
Fix $n \geq 1$ and $\nu >0$. To simplify notation, we write $\mathbb P$ instead of $\mathbb P_n ^\nu$. We define $\overline{\mathbb P}$ as a random triangulation distributed according  law $\mathbb P$ with a marked uniform edge. That is, for any triangulation of the sphere $T$ with $3n$ edges and any edge $e$ of $T$, we set
\[
\overline{\mathbb P} (T,e) = \frac{\nu^{m(T)}}{3n \cdot [t^{3n}] \gtr (\nu,t)}.
\]
Denote $\delta$ the root vertex and $e$ the marked edge of a triangulation sampled according to $\overline{\mathbb P}$. We have:
\begin{align}
\overline{\mathbb P} \left( \delta \in e \right) & = \sum_{k=1}^{3n}  \overline{\mathbb P} \left( \delta \in e \middle| \mathrm{deg}(\delta) = k \right) \cdot \overline{\mathbb P} \left( \mathrm{deg}(\delta) = k \right) \notag \\
& \geq \sum_{k=1}^{3n} \frac{k}{2 \cdot 3n} \overline{\mathbb P} \left( \mathrm{deg}(\delta) = k \right) \notag\\
& = \frac{1}{6n} \overline{\mathbb E} \left[ \mathrm{deg}(\delta) \right] = \frac{1}{6n} {\mathbb E} \left[ \mathrm{deg}(\delta) \right] \label{eq:tightness}
\end{align}
where we used the fact that an edge adjacent to the root vertex can contribute to its degree by at most $2$.
\begin{figure}[!ht]
\begin{center}
\includegraphics[width=0.9\textwidth]{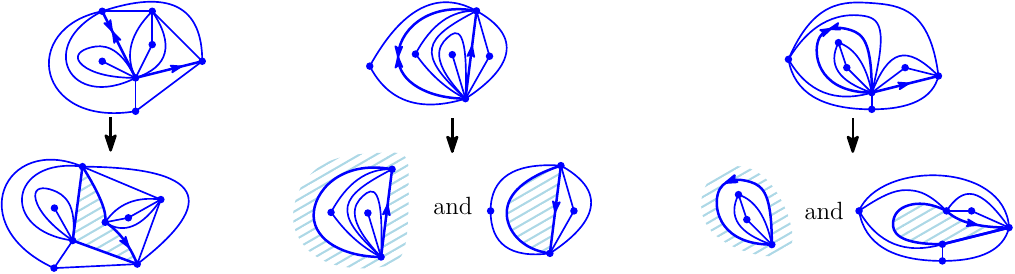}
\caption{\label{fig:OpenMarkedEdge}Some of the possibilities that may arise in the course of the proof of Lemma~\ref{lem:root-tight}. The 2 possibilities when neither the marked edge nor the root edge are loops (left and middle) and the case where only one of them is a loop (right). The double arrow indicates the marked edge and the $p$-gons are dashed.}
\end{center}
\end{figure}
Now, by duplicating and opening the marked edge and the root edge (see Figure \ref{fig:OpenMarkedEdge}), we can see that there is an injection from the set of triangulations of size $3n$ with a marked edge adjacent to the root vertex into triangulations with no marked edges. More precisely, we have the following cases when cutting along the two edges:
\begin{itemize}
	\item Both edges are not loops. We get either a triangulation of the $4$-gon or a pair of triangulations of the $2$-gon if the edges have the same endpoints.
	\item Both edges are loops. We get either a pair of triangulations of the $1$-gon if the marked edge is the root edge or a triplet of triangulations otherwise (two of the $1$-gon and one of the $2$-gon).
	\item One edge is a loop and not the other. We get a pair of triangulations, one of the $1$-gon and one of the $3$-gon.
\end{itemize}
Therefore, taking into account every case and the possible creation of new monochromatic edges, we have, 
\[
\overline{\mathbb P} \left( \delta \in e \right) \leq \max \left\{ \frac{1}{\nu},1 \right\}^2 \frac{[t^{3n+2}] (Z_4 + Z_2^2 + Z_1^2 + Z_1^2 Z_2 + Z_1 Z_3)}{3n \, [t^{3n}] \mathcal Z} = \mathcal O (1/n)
\]
from Theorem~\ref{theo:combimain}.

Together with equation \eqref{eq:tightness}, this yields that $\mathbb E \left[ \mathrm{deg}(\delta) \right]$ is bounded with $n$ giving the tightness of the sequence of random variables.
\end{proof}

To go from the tightness of the root degree to the tightness of the maximal degree in balls, we need some sort of invariance of the root degree by re-rooting. We will see in the next section that in fact, the distribution of the maps themselves are invariant under rerooting along a simple random walk, which is more than we need (see Lemma \ref{lem:invrootRW}).

\subsection{Invariance along a simple random walk and tightness}
\label{sec:invRW}

To formally introduce an invariance property by rerooting, we need some additional notation. Let $T$ be a rooted triangulation with spins (finite or infinite) and denote by $e_0$ the oriented root edge. A simple random walk on $T$ is an infinite random sequence $(e_0,e_1, \ldots )$ of oriented edges of $T$ defined recursively as follows. Conditionally given $(e_i, 0 \leq i \leq k)$, we let $e_{k+1}$ be an oriented edge whose origin is the endpoint $e_k^+$ of $e_k$ and whose endpoint is chosen uniformly among the $\mathrm{deg} (e_k^+)$ possible choices. We call the sequence $(e_0,e_1, \ldots )$ a simple random walk on $T$ started at the root edge, and denote by $P_T$ its law. Finally, if $e$ is an oriented edge of $T$, we denote by $T^{(e)}$ the triangulation $T$ re-rooted at $e$.

% For any pair $(T,(e_0,e_1,\ldots))$ consisting of a triangulation of $\mathcal T$ and a simple walk started at the root edge, we can define the shift operator $\Theta$ by
% \[
% \Theta (T,(e_0,e_1,\ldots)) = (T^{(e_1)},(e_1,e_2, \ldots)).
% \]
%Therefore,
If $\lambda$ is a probability distribution on $\mathcal T$, we denote by $\Theta ^{(k)} (\lambda)$ the distribution of a random triangulation sampled according to $\lambda$ and re-rooted at the $k$-th step of a simple random walk
% with a slight abuse of notation by forgetting the rest of the walk.
It is defined by
\[
\Theta ^{(k)} (\lambda) (A) = \int_{\mathcal T} \lambda (\mathrm d T) \,
\int P_T ( \mathrm d (e_0,e_1,\ldots)) \, \mathbf{1}_{\{ T^{(e_k)} \in A\}}
\]
for every Borel subset $A$ of $\mathcal T$. We invite the interested reader to check the work of Aldous and Lyons \cite{AL} where this framework is introduced for any unimodular measure (see also \cite{BC,Curien, CMM} for related discussions specific to random maps).

The following lemma is an easy adaptation of \cite[Proposition 19]{CMM} and its proof is \emph{mutatis mutandis} the same. See also \cite[Theorem 3.2]{AngelSchramm} for an analogous statement with a slightly different proof. We insist on the fact that this result holds independently of the tightness or convergence of the measures $\mathbb P_n^\nu$.

\begin{lemm} \label{lem:invrootRW}
The laws $\mathbb P_n^\nu$ and any of their subsequential limit $\mathbb P^\nu$ are invariant under re-rooting along a simple random walk in the sense that for every $k \geq 0$ and any $n \in \mathbb{N}\cup \{\infty\}$ we have $\Theta ^{(k)} (\mathbb P_n^\nu) = \mathbb P_n^\nu$.
\end{lemm}

\begin{proof}
See the proof of \cite[Proposition 19]{CMM}, which carries word for word in our setting.
\end{proof}

\begin{prop} \label{prop:tightness}
The family of probability measures $\{ \mathbb P_n^\nu \}$ is tight.
\end{prop}

\begin{proof}
The original proof of Angel and Schramm \cite[Lemma 4.4 page 22]{AngelSchramm} can also be adapted effortlessly to our setting with the help of Lemma~\ref{lem:invrootRW} that provides invariance of the laws under $\Theta^{(k)}$ and of Lemma~\ref{lem:root-tight} that provides the tightness of the law of the root degree.
\end{proof}

\subsection{Proof of Theorem~\ref{th:localCV}}

As in the previous section, thanks to the behavior of our generating series, things are not much more complicated than in the uniform setting and we can follow the original approach of Angel and Schramm \cite{AngelSchramm}. Recall the definition of \emph{rigid} triangulations (see \cite[Section 4.2]{AngelSchramm} for details), which are triangulations with holes such that one cannot fill the holes in two different ways to obtain the same triangulation of the sphere. First we show that subsequential limits of the $\mathbb P_n ^\nu$'s share common properties. This Proposition is analogous to \cite[Corollary 3.4 and Proposition 4.10]{AngelSchramm} and the proofs are almost identical, so we only give the main arguments.

\begin{prop} \label{prop:limit}
Every subsequential limit $\mathbb P^\nu$ of $(\mathbb P_n ^\nu)_{n \geq 1}$ has almost surely one end. In addition, for every finite rigid triangulation $\Delta$ with $\ell \geq 1$ holes without common edges and respective boundary conditions given by $\omega^{(1)},\ldots,\omega^{(\ell)} \in \{\ps,\ns\}^+$, we have:
\begin{equation}
\label{eq:limitball}
\mathbb P^\nu \left( \Delta \subset T \right)
= \frac{t_\nu^{|\Delta| - |\omega|} \nu^{m(\Delta) - m(\omega)}}{\kappa} \left( \prod_{j=1}^\ell Z_{\omega^{(j)}}(\nu,t_\nu) \right) \sum_{i=1}^\ell \frac{\kappa_{\omega^{(i)}}}{t_\nu^{|\omega^{(i)}|}Z_{\omega^{(i)}}(\nu,t_\nu)},
\end{equation}
where  $\omega=(\omega^{(1)},\ldots,\omega^{(\ell)})$ and $m(\omega)$ denotes the total number of monochromatic edges of the boundaries $\omega^{(1)} , \ldots , \omega^{(\ell)}$. We recall that the constants $\kappa$ and $\kappa_{\omega^{(i)}}$ are defined in Theorem~\ref{theo:combimain}.

Moreover, the probability that the $i$-th hole contains the infinite part of the triangulation is proportional to the $i$-th term in the sum.
\end{prop}

\begin{proof}
First, the one-endedness is an easy adaptation of Lemma 3.3 and Corollary 3.4 of \cite{AngelSchramm}. Indeed, if a subsequential limit has more than one end then, under this law, there exists $k > 0$ and $\varepsilon > 0 $ such that there is a simple cycle with $k$ edges containing the root that separates the triangulation into two infinite parts with probability larger than $\varepsilon$. This in turn means that for any integer $A$ and infinitely many $n$, the probability under $\mathbb P_n^\nu$ to have a simple cycle with $k$ edges containing the root that separates the triangulation into two parts with at least $A$ edges each is larger than say $\varepsilon /2$. Denote by $L(k,A)$ such an event. Its probability under $\mathbb P_n^\nu$ is bounded by
\[
\mathbb P_n^\nu \left( L(k,A) \right) \leq \sum_{|\omega| =k}\frac{1}{[t^{3n}] \mathcal Z (\nu, t)}
\sum_{\substack{3n_1+3n_2 = 3n + k\\3n_1,3n_2 \geq A+k}}
\nu^{-m(\omega)}
[t^{3n_1-k}] \big(Z_\omega (\nu,t)\big) \cdot [t^{3n_2-k}] \big(Z_\omega (\nu,t)\big)
\]
where the first sum is to fix the spins of the loop and the term $\nu^{-m(\omega)}$ is there to avoid counting monochromatic edges of the cycle twice and the number of edges on each side of the cycle is $3n_i -k$ including the boundary. From Theorem~\ref{theo:combimain}, we know that the coefficients in the above identity all share the same asymptotic behavior and we have, if $A$ is large enough and for some constant depending only on $\nu$ and $k$:
\begin{equation} \label{eq:oneend}
\mathbb P_n^\nu \left( L(k,A) \right) \leq \mathrm{Cst}
\sum_{\substack{3n_1+3n_2 = 3n + k\\3n_1,3n_2 \geq A+k}}
\frac{(t_\nu)^{-3n_1} n_1^{-\alpha} (t_\nu)^{-3n_2} n_2^{-\alpha}}{(t_\nu)^{-3n} n^{-\alpha}}
\end{equation}
with $\alpha$ being $5/2$ or $7/3$ depending on $\nu$. A classical analysis of the right hand side of \eqref{eq:oneend} shows that this probability is of order $\mathcal O \left( A^{-\alpha +1} \right)$ and thus goes to $0$, meaning that the triangulation cannot have more than one end.

The second statement is a straightforward computation.
Indeed, by decomposing triangulations $T$ such that $\Delta \subset T$ into $\Delta$ a triangulations with respective boundaries $\omega^{(1)} , \ldots , \omega^{(\ell)}$ and avoiding counting edges on the boundary of $\Delta$ twice we get
\begin{align*}
\mathbb P_n^\nu \left( \Delta \subset T \right)
& = 
\nu^{m(\Delta) - m(\omega)}
\frac{ [t^{3n - |\Delta| + |\omega|}]
\left( \prod_{j=1}^\ell Z_{\omega^{(j)}}(\nu,t) \right)}
{[t^{3n}] \mathcal Z (\nu,t)}\\
& = 
\nu^{m(\Delta) - m(\omega)}
\frac{ [t^{3n - |\Delta| + 2 |\omega|}]
\left( \prod_{j=1}^\ell t^{|\omega^{(j)}|}Z_{\omega^{(j)}}(\nu,t) \right)}
{[t^{3n}] \mathcal Z (\nu,t)}.
\end{align*}
In particular, we used the fact that the holes of $\Delta$ have no common edges. This hypothesis can be avoided by changing the definition of $m(\omega)$ but we prefer not to in order to keep technicalities at a minimum level.

The series $t^{|\omega_i|} Z_{\omega^{(i)}}(\nu,t)$ are all Ising-algebraic and have the following singular expansion at their singularity
\begin{align*}
t^{|\omega^{(i)}|} &Z_{\omega^{(i)}}(\nu,t)\\&= t^{|\omega^{(i)}|}_\nu Z_{\omega^{(i)}}(\nu,t_\nu) + B_{\omega^{(i)}}(\nu) \left( 1 - (t/t_\nu)^3 \right) + C_{\omega^{(i)}}(\nu) \left( 1 - (t/t_\nu)^3 \right)^{\alpha - 1} + o \left( \left( 1 - (t/t_\nu)^3 \right)^{\alpha - 1} \right) 
\end{align*}
with $\alpha(\nu) = 5/2$ or $7/3$ depending on $\nu$. Therefore, the product $\prod_{j=1}^\ell t^{|\omega^{(j)}|}Z_{\omega^{(j)}}(\nu,t)$ is also Ising algebraic with a singular expansion at its unique singularity of the form
\begin{align*}
\prod_{j=1}^\ell & t^{|\omega^{(j)}|}Z_{\omega^{(j)}}(\nu,t) \\
&=
 t^{|\omega|}_\nu \prod_{j=1}^\ell Z_{\omega^{(i)}}(\nu,t_\nu)
 + B_{\omega}(\nu) \left( 1 - (t/t_\nu)^3 \right) + C_{\omega}(\nu) \left( 1 - (t/t_\nu)^3 \right)^{\alpha - 1} + o \left( \left( 1 - (t/t_\nu)^3 \right)^{\alpha - 1} \right)
\end{align*}
with
\[
C_\omega (\nu) = \sum_{i = 1}^\ell C_{\omega^{(i)}}   \prod_{j \neq i} t^{|\omega^{(j)}|}_\nu Z_{\omega^{(j)}}(\nu,t_\nu).
\]
From there, it is straightforward to obtain
\begin{align*}
\mathbb P_n^\nu \left( \Delta \subset T \right)
&\sim_{n \to \infty}
\nu^{m(\Delta) - m(\omega)}
\sum_{i=1}^\ell
\frac{
\left( \prod_{j\neq i} t_\nu^{|\omega^{(j)}|}Z_{\omega^{(j)}}(\nu,t_\nu) \right) \cdot \kappa_{\omega^{(i)}} t_\nu^{-(3n - |\Delta| + 2 |\omega|)} (3n - |\Delta| + 2 |\omega|)^{-\alpha}}
{ \kappa(\nu) t_\nu^{-3n} n^{-\alpha}}\\
\end{align*}
with $\alpha(\nu) = 5/2$ or $7/3$ depending on $\nu$. This in turn yields
\[
\mathbb P_n^\nu \left( \Delta \subset T \right)
\to_{n \to \infty}
t_\nu^{|\Delta| - |\omega|}
\nu^{m(\Delta) - m(\omega)}
\sum_{i=1}^\ell
\frac{
\left( \prod_{j\neq i} Z_{\omega^{(j)}}(\nu,t_\nu) \right) \cdot \kappa_{\omega^{(i)}} t_\nu^{-|\omega^{(i)}|}}
{ \kappa(\nu)}
\]
finishing the proof of \eqref{eq:limitball}.

Fix $i \in \{1 , \ldots , \ell \}$. If  $N_i + N'_i =  3n - |\Delta| + |\omega|$, the probability under $\mathbb P_n^\nu$  that the i-th hole has $N_i$ edges and the other holes have a total of $N'_i$ edges is given by
\[
\frac{
\displaystyle{
[t^{N_i}]  Z_{\omega^{(i)}}(\nu,t) \, 
[t^{N'_i}]
\left( \prod_{j\neq i} Z_{\omega^{(j)}}(\nu,t) \right)}
}
{[t^{3n}] \mathcal Z (\nu,t)} \to_{n \to \infty}
\nu^{m(\Delta) - m(\omega)}
\frac{
\kappa_{\omega^{(i)}}}
{ t_\nu^{|\omega|-|\Delta| +|\omega^{(i)}| - N'_i}\kappa(\nu)}  \, [t^{N'_i}] \left( \prod_{j\neq i} Z_{\omega^{(j)}}(\nu,t) \right)
\]
when $N'_i$ is fixed. Summing over $N'_i$ finishes the proof of the proposition.
\end{proof}
Theorem \ref{th:localCV} now follows directly from the tightness of the laws $\mathbb P_n^\nu$ (Proposition \ref{prop:tightness}) and from Proposition \ref{prop:limit} which implies that the sequence has a unique possible subsequential limit.

\section{Basic properties of the limit}

We introduce another probability distribution on the set of finite triangulations, denoted $\pB$ and called the Boltzmann law. This probability measure is often found to be of central importance in local limits of planar maps. For example it appears in the limiting law of uniform triangulations without spins in \cite{AngelSchramm} where it is called the free distribution, or in \cite{CLGpeel}.
\begin{defi}
The critical Boltzmann distribution $\pB$ is a probability measure on the set of finite triangulations defined by
\[
\pB ( \{T \}) = \frac{\nu^{m(T)} t_\nu^{|T|}}{\mathcal Z(\nu,t_\nu)}.
\]
for all $T\in \tr[f]$ (recall that $\gtr(\nu,t_\nu)$ is finite thanks to Theorem~\ref{theo:combimain}.
We will always denote by $\mathbf{T}_{\bol}$ a Boltzmann triangulation of the sphere, that is a random finite triangulation of the sphere with law $\pB$.

For any finite word $\omega$ on $\{\ps,\ns\}$, define similarly the probability measure $\pB^\omega$ on $\tr[f]^\omega$ by setting
\[
\pB^\omega ( \{ T \} ) = \frac{\nu^{m(T)} t_\nu^{|T|}}{Z_{\omega}(\nu,t_\nu)}.
\]
for any $T\in \kT^{\omega}$. We call a random triangulation with law $\pB^{\omega}$ a Boltzmann triangulation with boundary condition $\omega$ and denote it $\mathbf{T}_{\bol}^{\omega}$.
\end{defi}

Boltzmann triangulations satisfy the following spatial Markov property:

\begin{prop}[Spatial Markov property for Boltzmann triangulations]
\label{prop:spacialMarkov}
For any finite and rigid triangulation $K$ with $p \geq 1$ holes without common edges and respective boundary conditions given by $\omega^{(1)} , \ldots , \omega^{(p)}$, we have:
\[
\pB \left( K \subset \mathbf{T}_{\bol} \right) = 
\frac{\nu^{m(K) - m(\omega)} t_\nu^{|K|-|\omega|}}{\mathcal Z(\nu, t_\nu)}
\prod_{i=1}^{p} Z_{\omega^{(i)}} (\nu, t_\nu)
\]
where $\omega = (\omega^{(1)} , \ldots , \omega^{(p)})$. 
 
In addition, conditionally on the event $\left\{ K \subset \mathbf{T}_{\bol} \right\}$, the parts of $\mathbf{T}_{\bol}$ filling each hole of $K$ are independent random triangulations with a boundary, distributed as Boltzmann triangulations with respective boundary conditions given by $\omega$.
\end{prop}
\begin{proof}
This is a straightforward computation, analogous to the one performed in the proof of Proposition~\ref{prop:limit}. Indeed, a finite triangulation $T$ such that $K \subset T$ can be decomposed into $K$ and a collection of triangulations with respective boundary conditions $\omega^{(i)}$. This yields:
\begin{align*}
\pB \left( K \subset \mathbf{T}_{\bol} \right) & = \sum_{T_1 \in \tr[f]^{\scriptscriptstyle{\omega^{(1)}}} , \ldots , T_p \in \tr[f]^{\omega^{(p)}}} \frac{\nu^{m(K) + m(T_1) + \cdots + m(T_p) - m(\omega)} t_\nu^{|K| - |\omega| + |T_1| + \cdots + |T_p| }}{\mathcal Z(\nu, t_\nu)} \\
& = \frac{\nu^{m(K) - m(\omega)} t_\nu^{|K|-|\omega|}}{\mathcal Z(\nu,t_\nu)}
\prod_{i=1}^{p} Z_{\omega^{(i)}} (\nu, t_\nu)
\end{align*}
proving the first claim.

To prove the second claim, fix $T_1, \ldots , T_p$ some finite triangulations with respective boundary conditions $\omega^{(1)} , \ldots , \omega^{(p)}$. Then:
\begin{align*}
\pB &\left( \mathbf{T}_{\bol} \setminus K = (T_1, \ldots, T_p ) \middle| K \subset \mathbf{T}_{\bol} \right)\\
& \qquad= \frac{\nu^{m(K) + m(T_1) + \cdots + m(T_p) - m(\omega)} t_\nu^{|K| - |\omega| + |T_1| + \cdots + |T_p| }}{\mathcal Z(\nu, t_\nu)} \left( \pB \left( K \subset \mathbf{T}_{\bol} \right) \right)^{-1}\\
& \qquad = \prod_{i=1}^p \frac{\nu^{m(T_i)} t_\nu^{|T_i|}}{Z_{\omega^{(i)}} (\nu,t_\nu)} = \prod_{i=1}^p \pB^{\omega^{(i)}} (T_i),
\end{align*}
which concludes the proof.
\end{proof}

Proposition \ref{prop:spacialMarkov} allows to interpret the ball probabilities \eqref{eq:limitball} as an absolute continuity relation between $\mathbb P_\infty$ and $\pB$. Indeed, for $\Delta$ a ball of radius $r$ of some finite triangulation and $\omega = (\omega^{(1)}, \ldots ,\omega^{(\ell)})$ its boundary words, this probability can be written as:
\begin{equation} \label{eq:abscontprob}
\mathbb P_\infty \left( B_r (T) = \Delta \right) = \left( \frac{\mathcal Z (\nu, t_\nu)}{\kappa} \cdot \sum_{i=1}^\ell \frac{\kappa_{\omega^{(i)}}}{t_\nu^{\omega_i} Z_{\omega^{(i)}}(\nu, t_\nu)} \right) \cdot \mathbb P_{\bol} \left( B_r (T) = \Delta \right).
\end{equation}
This observation motivates the following definition:
\begin{defi}
Fix a finite triangulation $T$ and $r\geq 0$. If $B_r (T)$ is the whole triangulation $T$, set $\omega_r(T)=\emptyset$. Otherwise  set $\omega_r(T) = (\omega^{(1)},\ldots , \omega^{(\ell(T,r))})$ to be the spin configurations on the boundary of $B_r(T)$. We define:
\begin{equation}\label{eq:MR}
\mathbf M_r (T) = \begin{cases}
\displaystyle{
\frac{\mathcal Z (\nu, t_\nu)}{\kappa} \cdot \sum_{i=1}^{\ell(T,r)} \frac{\kappa_{\omega^{(i)}}}{t_\nu^{\omega^{(i)}} Z_{\omega^{(i)}}(\nu, t_\nu)}
}
& \text{if $\omega_r (T) \neq \emptyset$},\\
0 & \text{if $\omega_r (T) = \emptyset$}.
\end{cases}
\end{equation}
\end{defi}

Inspired by \cite[Theorem 4]{CLGpeel}, formula \eqref{eq:abscontprob} can be directly reformulated as follows:
\begin{prop}\label{th:Martingale}
The random process $(\mathbf M_r (\mathbf T_{\bol}) )_{r \geq 0}$ is a martingale with respect to the filtration generated by $ (B_r (\mathbf T_{\bol}))_{r \geq 0}$. Moreover, if $F$ is any nonnegative measurable function on the set of triangulations with holes, we have for every $r\geq 1$
\begin{equation} \label{eq:abscont}
\mathbb E \left[ F( B_r(\mathbf T_\infty ) \right] = \mathbb E \left[ \mathbf M_r F( B_r(\mathbf T_{\bol} ) \right]
\end{equation}
\end{prop}

We conclude this section by stating the spatial Markov property for the IIPT. First, we need to introduce the analog of the IIPT for triangulations with fixed boundary condition. Let $\omega$ be a non empty word on $\{\ps,\ns \}$. We can define the probability measure $\pn ^{\omega}$ on $\mathcal{T}_{3n - |\omega|}^{\omega}$ by
\[
\pn ^{\omega} (\{T\}) = \frac{\nu^{m(T)}}{[t^{3n - |\omega|}] Z_\omega(\nu,t) }.
\]
A slight modification of the proof of Theorem \ref{th:localCV} shows that the sequence $(\pn ^{\omega})_{n \geq 1}$ converges weakly in $(\mathcal{T}^{\omega},\dloc)$ to a probability measure supported on one-ended infinite triangulations with boundary condition $\omega$. We denote this limiting probability measure by $\plim^{\omega}$ and call it the law of the Ising Infinite Planar Triangulation with boundary condition $\omega$. As in the uniform setting, this law appears naturally in the spatial Markov property of the IIPT:

\begin{prop}[Spatial Markov property for the IIPT]\label{prop:spatialIIPT}
Fix $K$ a finite rigid triangulation with $\ell$ holes (the holes can have common vertices but have no common edges) and endowed with a spin configuration such that the boundary conditions of its holes are given by $\omega=(\omega^{(1)},\ldots,\omega^{(\ell)})$.
On the event $\left\{ K\subset \mathbf{T}_{\infty} \right\}$, let us denote by $\mathbf{T}_i$ the component of $\mathbf{T}_{\infty}$ inside the i-th hole of $K$. Then, almost surely, only one of these components is infinite and the probability that it is $\mathbf{T}_i$ is given by:
\[
\plim \left\{ K\subset \mathbf{T}_{\infty} , \, \text{$\mathbf{T}_i$ is infinite} \right\} = \frac{t_\nu^{|K| - |\omega|} \nu^{m(K) - m(\omega)}}{\kappa(\nu)} \left( \prod_{\substack{j=1\\ j \neq i}}^\ell Z_{\omega^{(j)}}(\nu,t_\nu) \right) \kappa_{\omega^{(i)}}(\nu) t_\nu^{-|\omega^{(i)}|}.
\]
Finally, if we fix $i \in \{1, \ldots, \ell \}$, conditionally on the event $\left\{ K\subset \mathbf{T}_{\infty} , \, \text{$\mathbf{T}_j$ is finite for $j \neq i$} \right\}$
\begin{enumerate}
\item The random triangulations with boundary conditions $(\mathbf{T}_j)_{1 \leq j \leq \ell}$ are independent;
\item The random triangulation $\mathbf{T}_i$ is distributed as the IIPT with boundary condition $\omega^{(i)}$;
\item  For $j \neq i$, the random triangulation $\mathbf{T}_j$ is distributed as a Boltzmann triangulation with boundary condition $\omega^{(j)}$.
\end{enumerate}
\end{prop}
\begin{proof}
Everything follows directly from Proposition~\ref{prop:limit}.
\end{proof}

\section{The critical IIPT is almost surely recurrent}

\subsection{Generating series of triangulations with simple boundary}\label{sub:seriesCrit}

We start with a technical lemma about the generating series of triangulations with simple boundary. For every $p>0$, we set $\kappa_p = \sum_{|\omega| = p} \kappa_{\omega}$. Since we only use this Lemma to prove Theorem~\ref{th:recurrence} and since our proof of this Theorem does not work for all $\nu$ (see Remark~\ref{rem:nurec}), we restrict ourselves to $\nu=\nu_c$ for the sake of simplicity.
\begin{lemm} \label{lem:growthCm}
At $\nu=\nu_c$, the series $\sum_{p \geq 1} \kappa_p y^p$ has radius
of convergence {$y_c$} with $y_c =
\frac{3}{5} (1+ \sqrt{7})$.
\end{lemm}
\begin{proof}
Let $A_p=\sum_{\omega,|\omega|=p}A_\omega(\nu_c)$ (resp. $B_p$,
$C_p$), so that, in view of Theorem~\ref{theo:combimain},
$t^p Z_p(\nu_c,t)=\sum_{\omega,|\omega|=p}t^p Z_\omega(\nu_c,t)$ is
Ising-algebraic with parameters $A_p$, $B_p$ and $C_p$
  \[
 (t^pZ_p(\nu_c,t))\mid_{t^3=t_{\nu_c}^3(1-x)}= A_p+ B_p\cdot x
  + C_p\cdot x^{4/3}(1+\varepsilon_p(x)),
  \]
  with $\varepsilon_p(x)$ holomorphic in the open disc $\mathcal{D}$
  of radius 1 centered at $x=1$, and such that $\varepsilon_p(x)\to0$
  as $x\to0^+$.
  While the latter expansions are \emph{a priori} non uniform
  (in the sense that $\varepsilon_p(x)$ depends on $p$), we can still
  define the formal power series
  \[
  A(y)=\sum_{p\geq1}A_py^p,
  \quad
   B (y)=\sum_{p\geq1} B_py^p,
  \quad\textrm{ and }\quad
   C(y)=\sum_{p\geq1} C_py^p,
  \]
  and wonder about their
  respective radii of convergence. 

  Our strategy to determine these radii of convergence is to relate
  $A(y)$, $ B(y)$ and $ C(y)$ to the formal power
  series $ Z(t,ty) $ of $\mathbb{Q}[[y,t^3]]$ defined as
  \[
  Z(t,ty)
  =\sum_{p\geq1}t^pZ_{p}(\nu_c,t)\cdot y^p=\sum_Ty^{p(T)}t^{p(T)+|T|}\nu_c^{m(T)}.
  \]
  More precisely we would like to view $Z(t,ty)$ as an analytic
  function of $t^3$ with $y$ a complex parameter fixed in an appropriate
  domain and to perform an expansion as $t^3 \to t_{\nu_c}^3$ of the form
  \begin{equation*}
  {{Z}}(t,ty)={{A}}(y)+{B}(y)x+{{C}}(y)x^{4/3}(1+\varepsilon(y,x))
  \end{equation*}
  with $t^3=t_{\nu_c}^3(1-x)$ and $\lim_{x\to 0^+}\varepsilon(y,x)=0$.

  Now the series $ Z(t,ty)$ is the unique formal power series
  solution of an explicit algebraic equation
  \begin{equation}\label{eq:PZt}
  P(Z(t,ty),y,U(t))=0
  \end{equation}
  for some polynomial $P(z,y,u)$ of degree 4 in $z$, and $U(t)$ is the
  series introduced in Theorem~\ref{th:gsBoundary}: this equation can
  be deduced from \cite[Lemma 31]{BernardiBousquet} using the fact
  that our $ Z(t,ty)$ is exactly the series $R(0,ty)$ there.  In
  particular we shall consider the unique formal power series 
  $\zeta(y,u)$ solution of the equation
  \begin{equation}\label{eq:PZU}
    P(\zeta(y,u),y,u)=0
  \end{equation}
  so that, as formal power series, $ Z(t,ty)=\zeta(y,U(t))$.

  In particular $Z(t,ty)$ as a complex function can be
  identified near the origin $(t,y)=(0,0)$ with the branch of the
  analytic variety defined by $P$ having the expected Taylor
  expansion. We will use this to prove in Lemma~\ref{lem:analytic}
  that $Z(t,y)$ is analytic in a polydisc
  $\mathcal{D}(0,t_{\nu_c})\times\mathcal{D}(0,y_c)$ and singular at
  the point $(t_{\nu_c},y_c)$, where $y_c$ is as in
  Lemma~\ref{lem:growthCm}.

  On the other hand using Equation~\eqref{eq:PZt} (or
  Equation~\eqref{eq:PZU}) we can study each branch
  $\tilde{{Z}}(t,ty)$ of the analytic variety defined by $P$
  near the point $t=t_{\nu_c}$, and derive explicit descriptions of
  their coefficients in an expansion
  \begin{equation*}
  \tilde{{Z}}(t,ty)
  =\tilde{{A}}(y)
  +\tilde{{B}}(y)x
  +\tilde{{C}}(y)x^{4/3}(1+\varepsilon(y,x))
  \end{equation*}
  with $t^3=t_{\nu_c}^3(1-x)$.
  Indeed taking $t=t_{\nu_c}$ in Equation~\eqref{eq:PZt} or
  $u=U_c:=U(t_{\nu_c})$ in Equation~\eqref{eq:PZU} we obtain an
  algebraic equation satisfied by $\tilde
  {{A}}(y)=\tilde{{Z}}(t_{\nu_c},t_{\nu_c}y)=\zeta(y,U_c)$. Moreover we
  will identify
  \begin{align*}
  \tilde {{B}}(y)&=\lim_{t\to t_c}(\tilde{{Z}}(t_{\nu_c},t_{\nu_c} y)-\tilde A(y))(t_c-t)^{-1}\quad\textrm{ and}\\
  \tilde {{C}}(y)&=\lim_{t\to t_c}(\tilde{{Z}}(t_{\nu_c},t_{\nu_c} y)-\tilde A(y)-\tilde B(y)(t_c-t))(t_c-t)^{-4/3}
  \end{align*}
  in terms of partial derivatives of $P$ evaluated at $z=\tilde {{A}}(y)$
  and $u=U(t_{\nu_c})$:
  \begin{align} \notag
    \tilde {{B}}(y)=\frac{P_B(\tilde {{A}}(y),y,U_c)}{\partial_{z}P(\tilde {{A}}(y),y,U_c)}\quad\textrm{ and }\quad
    \tilde {{C}}(y)=\frac{P_C(\tilde {{A}}(y),y,U_c)}{\partial_{z}P(\tilde {{A}}(y),y,U_c)}
  \end{align}
  for some explicit polynomial $P_B$ and $P_C$.

  Finally the positivity properties of $ Z(t,ty)$ allow to
  discriminate between the possible branches and characterize $
  A(y)$ as the unique formal power series in the variable $y$ such
  that
  \begin{align} \notag
    Q( A(y),y,U_c)=0
  \end{align}
  where $Q$ is a well chosen factor of the polynomial $P$.  In
  particular the equation for $ A(y)$ implies that $ A(y)$
  has $y_c$ as radius of convergence and from the irreducible rational
  expression of $ B(y)$ and $ C(y)$ we can check that no
  cancellation occur and these two series also have $y_c$ as radius of
  convergence. All computations are available in the companion Maple file \cite{Maple}.
\end{proof}

\begin{lemm}\label{lem:analytic}
  The series $Z(t,ty)$ is analytic in the larger domain
  $\mathcal{D}(0,t_{\nu_c}) \times \mathcal{D}(0,{y_c})$ with
  $y_c=\frac35(1+\sqrt7)$, and singular at $(t_{\nu_c},y_c)$.
\end{lemm}
\begin{proof}
  On the one hand, this formal power series is
  by definition an element of $\mathbb{Q}(\nu_c)[y][[t]]$, and for $|y|\leq1$
  the series is term-by-term dominated by the series
  $Z_3(\nu_c,|t|)$. Indeed, since $\nu_c>1$,
  \[
  |Z(t,ty)|\leq\sum_{T}|t|^{|\Delta(T)|}\nu_c^{m(T)}\leq\sum_{T}|t|^{|\Delta(T)|}\nu_c^{m(\Delta(T))}\leq Z_3(\nu_c,|t|),
  \]
  where $\Delta(T)$ denotes the triangulation obtained by
  triangulating the outer face of $T$ from a new vertex. Since $Z_3$
  has radius of convergence $t_{\nu_c}$ we already know that:
  \begin{itemize}
  \item
    $ Z(t,ty)$ is absolutely convergent in the polydisc
    $\mathcal{D}(0,t_{\nu_c}) \times \mathcal{D}(0,1)$.
  \end{itemize}
  For any fixed $y$, let $t_c(y)$ denote the radius of convergence of
  the series $ Z(t,ty)$ in the variable $t$.  In view of the
  positivity of the coefficients of $Z(t,ty)$, the function $t_c(y)$ is
  a weakly decreasing function of $y$ for $y$ positive, and it is at
  most equal to $t_{\nu_c}$ since $t_{\nu_c}$ is the radius of
  convergence of $t\cdot Z_1(\nu_c,t)=[y] Z(t,ty)$: in particular
  $t_c(y)=t_{\nu_c}$ for $y\in(0,1)$.

    For any $y>0$, $ Z(t,ty)$ is a series with positive
  coefficients, so by Pringsheim's theorem it must be singular at
  $t=t_c(y)$. In particular if $t_c(y)<t_{\nu_c}$, then $U(t)$ is
  regular in $\mathcal{D}(0,t_c(y))$ and, as a function of $u$, $\zeta(y,u)$ admits an
  analytic continuation in an open domain containing $(0,U(t_c(y)))$
  and it is singular at $u_c(y)=U(t_c(y))$. Recall from Lemma~\ref{lem:etU}  that $U(t)$ has
  nonnegative coefficients. Hence, it is an increasing function of $t$. Consequently, $u_c(y)$ must be a  nonincreasing function of
  $y$, with $u_c(y)=U(t_{\nu_c})$ for $y\in(0,1)$.

  As a consequence of the previous analysis we can look for $u_c(y)$
  among the decreasing branches in the root variety of the
  discriminant $\Delta(y,u)=\mathrm{discrim}_zP(z,y,u)$ with respect
  to $z$ of the polynomial $P(z,y,u)$. This discriminant factors into
  three irreducible factors of degree at most three in $y$, that can
  thus be explicitly analyzed: for $y<y_c=\frac35(1+\sqrt 7)$, all
  real positive branches have either $u>U(t_{\nu_c})$ or are
  increasing. At $y=y_c$, three discriminant branches meet with
  $u=U(t_{\nu_c})$ which is the minimal positive root of
  $\Delta(y_c,u)$. We therefore conclude that $t_c(y)=t_{\nu_c}$ for
  $y\in(0,y_c)$, or in other terms: for any fixed $y\in(0,y_c)$, the series $Z(t,ty)$ has
    radius of convergence $t_{\nu_c}$. It implies that the series is absolutely convergent in the polydisc
    $\mathcal{D}(0,t_{\nu_c}) \times \mathcal{D}(0,y_c)$, which concludes the proof.
\end{proof}

\subsection{Root degree distribution and recurrence}\label{sec:rootDegree}
 Since the IIPT is the local weak limit of uniformly rooted maps, by Gurel Gurevich and Nachmias \cite{GGN}, it is enough to prove that the root degree (i.e. the number of half-edges incident to the root) distribution of the IIPT has exponential tails.

To study this degree, let us have a look at the structure of the hull of radius $1$ around the root, see Figure \ref{fig:1hull} for an illustration. This hull, denoted by $\overline{B}_1(\mathbf T _\infty)$ is by definition the ball ${B}_1(\mathbf T _\infty)$ completed by the finite connected components of $\mathbf T _\infty \setminus {B}_1(\mathbf T _\infty)$. It is therefore a triangulation (with spins) with one hole, which corresponds to the part of $\partial{B}_1(\mathbf T _\infty)$ separating the root vertex from infinity in the map. Such maps (or more precisely slight modifications) are called triangulations of the cylinder, and have been extensively studied in \cite{CLGfpp,CuMe,Kr,Me} to which we refer for a more detailed analysis. In particular, each edge of $\partial \overline{B}_1(\mathbf T _\infty)$ belongs to a face of $\mathbf T _\infty$ having the root vertex as third vertex. The slots between two consecutive such faces are filled with independent Boltzmann triangulations with the proper boundary conditions. See Figure~\ref{fig:1hull} for an illustration.
\begin{figure}[!ht]
\begin{center}
\includegraphics[width=0.7\textwidth]{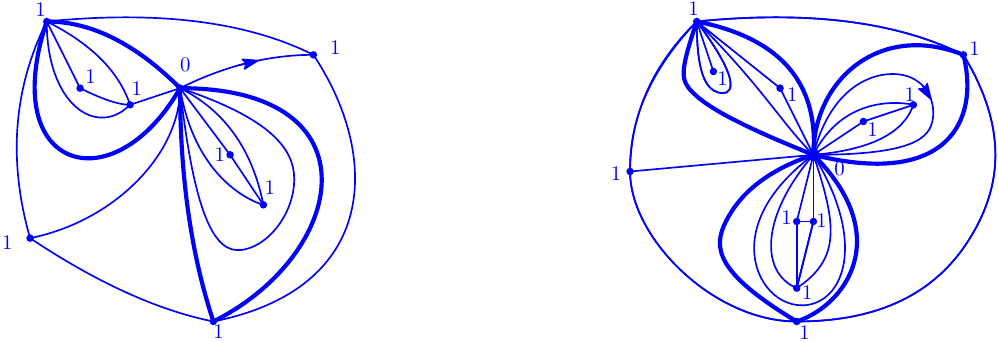}
\caption{\label{fig:1hull}Structure of the hull of radius $1$ around the root of the IIPT, when the root is not a loop (left) and when the root is a loop (right). Bold edges form the boundary of the slots.}
\end{center}
\end{figure}

The degree of the root vertex in $\mathbf T_ \infty$ is then the sum of the degrees of the root vertex of each of these Boltzmann triangulations filling the slots of the hull of radius $1$. Therefore, we only have to prove that the distribution of the root degree of these Boltzmann triangulations and the boundary length $|\partial \overline{B}_1(\mathbf T _\infty)|$ have exponential tails, which is done in Propositions \ref{prop:HullPerimTail} and \ref{prop:BoltzDegTail}.

\begin{prop} \label{prop:HullPerimTail}
There exist two constants $c>0$ and $\lambda<1$ such that for every $p \geq 1$,
\[
\mathbb P_\infty \left( |\partial \overline{B}_1(\mathbf T _\infty)| \geq p \right) \leq c \, \lambda^p.
\]
\end{prop}

\begin{proof}
As illustrated in Figure~\ref{fig:1hull} and described above, the hull of radius 1 can be decomposed into its faces sharing an edge with the boundary and slots. The slots are filled with Boltzmann triangulations of the 2-gon with boundary condition $(\ps,\ps)$ or $(\ps,\ns)$. Special care has to be taken if the root is a loop, then the slot containing it, is slightly different and can be decomposed into a Boltzmann triangulation of the 1-gon and a Boltzmann triangulation of the 3-gon with boundary condition $(\ps,\omega_1,\ps)$, where $\omega = \omega_1\ldots \omega_p$ gives the boundary condition of the hull of radius 1. 
The spatial Markov property stated in Proposition~\ref{prop:spatialIIPT} hence yields:
\begin{align*}
\mathbb P_\infty \left( |\partial \overline{B}_1(\mathbf T _\infty)| = p \right)
& =
\sum_{|\omega| = p} \mathbb P_\infty \left( \partial \overline{B}_1(\mathbf T _\infty) = \omega \right) \\
& = \frac{1}{\kappa}
\sum_{\omega=\omega_1 \cdots \omega_p} \left(\kappa_\omega t_{\nu_c}^{-|\omega|}
\Big(Z_{\ps\omega_1}(t_{\nu_c})+2Z_{\ps\ps\omega_1}(t_{\nu_c})Z_{\ps}(t_{\nu_c}))\Big)\Big(\prod_{i=2}^p Z_{\ps \omega_i}(t_{\nu_c})
\Big)
\right)\\
&\leq \mathrm{Cst} \cdot \kappa_p \cdot \left( \frac{Z_{\ps\ps}(t_{\nu_c}) \wedge Z_{\ps\ns}(t_{\nu_c})}{t_{\nu_c}} \right)^{p-1}
\end{align*}
where the constant does not depend on $p$.
And the result follows since
\[
\frac{Z_{\ps\ps}(t_{\nu_c}) \wedge Z_{\ps\ns}(t_{\nu_c})}{t_{\nu_c}}
=
\frac{131}{600} \frac{4-\sqrt{7}}{(50 \sqrt{7} - 110)^{1/3}} \simeq 0.105 <y_c = 3(1+\sqrt{7})/5
\]
from the value of $y_c$ given in Lemma \ref{lem:growthCm}.
\end{proof}

Let us now turn our attention to the root-degree of Boltzmann triangulations. 
\begin{prop} \label{prop:BoltzDegTail}
Let $\omega$ be a non-empty word on $\{\ps,\ns\}$ and $\mathbf D_\omega$ be the degree of the root of a Boltzmann triangulation with boundary condition $\omega$ and parameter $\nu=\nu_c$. Then, there exist two constants $c$ and $\lambda<1$, such that, for every $\omega$ and every $k \geq 1$,
\[
\mathbb P \left( \mathbf D_\omega \geq k \right) \leq c \, \lambda^k.
\]
% \color{red}{Adapter l'énoncé sur les $\nu$ o\`u la preuve marche.}
\end{prop}
\begin{rema} \label{rem:nurec}
The following proof of this proposition does not work for all values of $\nu$. However, numerical computations show that it should work for $\nu$ ranging from $0.3$ to $2.07$, therefore most of the proof is written for a generic $\nu$. The missing part to state the result for these values of $\nu$ is an argument to prove that the spectral radius of $M$ defined in \eqref{eq:meanmatrix} remains smaller than 1.
\end{rema}
\begin{proof}
In~\cite[Proposition 30]{CLGfpp}, a similar result is obtained for Boltzmann triangulations without spins (corresponding to the case $\nu=1$ in our setting). Following the same approach, we stochastically dominate the root degree by a subcritical branching process. To that aim, we explore a Boltzmann triangulation with a peeling process that will focus on exploring the neighbours of the root edge.

Fix $\omega$ a non empty word and let $\mathbf T_{\bol}^\omega$ be a Boltzmann triangulation with boundary condition $\omega=(\omega_1, \ldots , \omega_p)$. Recall that the root face of this triangulation lies on the right-hand side of its root edge. When the face adjacent to the left-hand side of the root face is revealed, several possibilities can occur. These possibilities are illustrated in Figure \ref{fig:BoltzPeel}  and their respective probabilities can be expressed in terms of the generating series of triangulations with simple boundary conditions evaluated at their radius of convergence. Let us enumerate them:
\begin{enumerate}
	\item If $p=2$, then  $\mathbf T_{\bol}^\omega$ may be reduced to the edge-triangulation. It happens with probability:
	\[
	\frac{\nu^{\delta_{\omega_1 = \omega_2}} t_\nu}{Z_\omega (\nu, t_\nu)}
	\]
	and the exploration stops if this event occurs.
	\item The third vertex of the revealed face is an inner vertex of $\mathbf T_{\bol}^\omega$ and has spin $a \in \{\ps,\ns\}$. This happens with probability
	\[
	\nu^{\delta_{\omega_1 = \omega_p}} t_\nu
	\frac{Z_{a\omega} (\nu, t_\nu)}{Z_\omega (\nu, t_\nu)}
	\]
	and the rest of the triangulation is distributed as $\mathbf T_{\bol}^{a\omega}$.
	\item The third vertex of the revealed face belongs to the boundary of $\mathbf T_{\bol}^\omega$, say the $i$-th starting from the target of the root edge, which has spin $\omega_i$. This happens with probability
	\[
	\nu^{\delta_{\omega_1 = \omega_p}} t_\nu
	\frac{
	Z_{(\omega_1, \ldots , \omega_i)}(\nu, t_\nu) \cdot Z_{(\omega_i, \ldots , \omega_p)} (\nu, t_\nu)
	}{Z_\omega (\nu, t_\nu)}
	\]
	and the two triangulations remaining to explore are independent and distributed respectively as $\mathbf T_{\bol}^{(\omega_1, \ldots , \omega_i)}$ and $\mathbf T_{\bol}^{(\omega_i, \ldots , \omega_p)}$. Since we are only interested in the root degree of $\mathbf T_{\bol}^{\omega}$, we can further distinguish two subcases:
	\begin{enumerate}
		\item The third vertex is not the root vertex (meaning $p >1$ and $i \neq p$). In that case only the subtriangulation distributed according to $\mathbf T_{\bol}^{(\omega_i, \ldots , \omega_p)}$ contains the root vertex of $\mathbf T_{\bol}^{\omega}$ and we discard the other remaining part.
		\item The third vertex is the root vertex (meaning $i = p$). In this case the two remaining subtriangulations contain the root vertex and we have to explore both of them. We will say that the exploration branches. Notice that in this case, the two subtriangulations are distributed as $\mathbf T_{\bol}^{\omega}$ and $\mathbf T_{\bol}^{\ps}$, and that the probability of this event simplifies into $\nu^{\delta_{\omega_1 = \omega_p}} t_\nu Z_\ps(\nu,t_\nu)$.
	\end{enumerate}
\end{enumerate}
\begin{figure}[!ht]
\begin{center}
\includegraphics[width=0.9\textwidth]{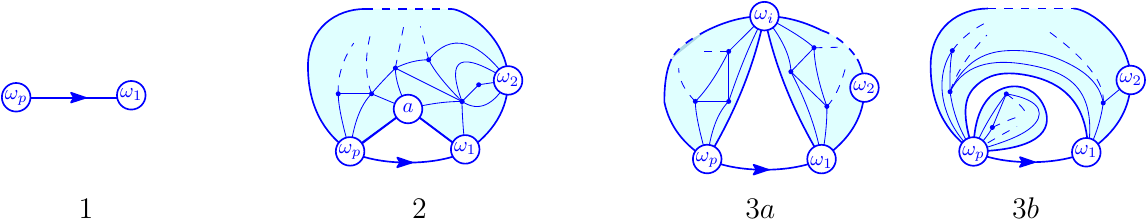}
\caption{\label{fig:BoltzPeel}The different cases that can occur in the branching exploration of a triangulation. The labels $1$, $2$, $3a$ and $3b$ refer to the description in the text.}
\end{center}
\end{figure}

When the exploration is complete, every edge adjacent to the root edge is discovered, and each exploration step (taking into account the steps of every branch of the exploration) increases the degree by $1$ or $2$ (for loops).
Therefore, the degree of the root vertex of $\mathbf T_{\bol}^{\omega}$ is bounded from above by twice the total number of particles in a multitype branching process $\mathcal B$ where the types of the particles are words in $\{\ps,\ns\}^{\mathbb N^\star}$ ending by $\ps$ (which is always the spin of the root vertex) and whose transition probabilities are given by:
\begin{itemize}
 \item Case 1: A particle of type $(a,\ps)$ has \textbf{no child} with probability
 \[
 \frac{\nu^{\mathbf 1 (a = \ps)} t_\nu}{Z_{a\ps} (\nu, t_\nu)}.
 \]
 \item Case 2: A particle of type $\omega=(\omega_1, \ldots,\omega_p)$ has \textbf{one child} of type $a\omega$ with probability
 	\[
	\nu^{\mathbf 1 (\omega_1 = \omega_p)} t_\nu
	\frac{Z_{a\omega} (\nu, t_\nu)}{Z_\omega (\nu, t_\nu)}.
	\]
 \item Case 3a: A particle of type $\omega=(\omega_1, \ldots,\omega_p)$ with $p >1$ has \textbf{one child} of type $(\omega_i,\ldots, \omega_p)$ with $i < p$ with probability
 	\[
	\nu^{\mathbf 1 (\omega_1 = \omega_p)} t_\nu
	\frac{
	Z_{(\omega_1, \ldots , \omega_i)} (\nu, t_\nu) \cdot Z_{(\omega_i, \ldots , \omega_p)} (\nu, t_\nu)
	}{Z_\omega (\nu, t_\nu)}.
	\]
 \item Case 3b: A particle of $\omega$ has \textbf{two children} of respective types $\ps$ and $\omega$ with probability
 \[
\nu^{\mathbf 1 (\omega_1 = \omega_p)} t_\nu Z_\ps(\nu,t_\nu).
 \]
\end{itemize}

The branching process $\mathcal B$ has an infinite number of types, which makes it difficult to analyze. We introduce another branching process, denoted $\mathcal B'$, that stochastically dominates $\mathcal B$, has only finitely many types (and as few as possible!) and is subcritical. Since only particles of type $\ps\ps$ and $\ns\ps$ can die and branching always give birth to particles of type $\ps$, we keep these three types and group types of length larger than two together. To get an interesting bound, we end up keeping five types, denoted $\ps, \ps\ps , \ns \ps, \Omega \ps \ps, \Omega \ns \ps$, where the last two regroup the corresponding original types of length larger than two.

The offspring distribution for type $\ps$ in $\mathcal B'$ is the same as in $\mathcal B$. Namely, an individual of type $\ps$ has:
 \begin{itemize}
 	\item \textbf{Two} children of type $\ps$ with probability
 $\displaystyle
\nu t_\nu Z_{\ps} (\nu,t_\nu)$.
 	\item \textbf{One} child of type $\ps\ps$ with probability $\displaystyle \nu t_\nu\frac{Z_{\ps\ps} (\nu, t_\nu)}{Z_\ps (\nu, t_\nu)}$.
\item \textbf{One} child of type $\ns\ps$ with probability $\displaystyle \nu t_\nu\frac{Z_{\ns\ps} (\nu, t_\nu)}{Z_\ps (\nu, t_\nu)}$.
 \end{itemize}
The offspring distribution for types  $\ps\ps,$ and $\ns \ps$ in $\mathcal B'$ are the same as in $\mathcal B$, where all particles of type length larger than two are merged. Namely, for $a$ fixed in $\{\ns,\ps\}$, an individual of type $a\ps$ has:
\begin{itemize}
  	\item No children with probability $\displaystyle \frac{\nu^{\mathbf 1 (a = \ps)} t_\nu}{Z_{a\ps} (\nu, t_\nu)}$.
 	\item \textbf{One} child of type $a\ps$ with probability $\displaystyle \nu^{\mathbf 1 (a = \ps)} t_\nu Z_{\ps} (\nu,t_\nu)$.
 	\item \textbf{Two} children of types $\ps$ and $a\ps$ with probability $\displaystyle \nu^{\mathbf 1 (a = \ps)} t_\nu Z_{\ps} (\nu,t_\nu)$.
	\item \textbf{One} child of type $\Omega a \ps$ with probability $\displaystyle 1-\frac{\nu^{\mathbf 1 (a = \ps)} t_\nu}{Z_{a\ps} (\nu, t_\nu)} - 2 \, \nu^{\mathbf 1 (a = \ps)} t_\nu Z_{\ps} (\nu,t_\nu)$.
\end{itemize}

Fix $a\in \{\ps,\ns\}$. We now turn our attention to individuals of type $\Omega a \ps$. Since only individuals of type $\ps\ps$ and $\ns\ps$ can die, we want the probability of giving birth to such particles to be smaller in $\mathcal B'$ than in $\mathcal B$. For $\omega  \in \{\ps,\ns\}^+$, the probability that a particle of type $\omega a \ps$ has a child of type $a \ps$ in $\mathcal{B}$ is equal to: 
\[
\nu^{\mathbf 1 (\omega_1 = \ps)} t_\nu
\frac{ Z_{\omega a} (\nu, t_\nu)\cdot Z_{a\ps} (\nu, t_\nu) }{Z_{\omega a\ps} (\nu, t_\nu)}.
\]
To get a lower bound (independent on $\omega$) for these probabilities, we use the functional equation of Proposition \ref{prop:PeelSimple} evaluated at $t=t_\nu$:
\begin{align*}
Z_{\omega a}(\nu, t_\nu)
&=
\frac{
\nu^{\mathbf{1} (a=\omega_1)} t_\nu
}
{1-2 \nu^{\mathbf{1}(a=\omega_1)} t_\nu Z_\ps(\nu, t_\nu)}
\left(
Z_{\omega a \ps}(\nu, t_\nu)+Z_{\omega a \ns}(\nu, t_\nu)
+
\sum_{\omega =\omega'c\omega''}Z_{\omega'c}(\nu, t_\nu)Z_{c\omega''a}(\nu, t_\nu)\right),
\end{align*}
which gives:
\[
\nu^{\mathbf 1 (\omega_1 = \ps)} t_\nu
\frac{Z_{\omega a} (\nu, t_\nu)\cdot Z_{a\ps} (\nu, t_\nu)}{Z_{\omega a\ps} (\nu, t_\nu)}
\geq
\frac{
(1 \wedge \nu)^2 t_\nu^{2} Z_{a\ps}(\nu, t_\nu)
}
{1-2 \, (1 \wedge \nu) \, t_\nu Z_\ps(\nu, t_\nu)}.
\]
Hence, we set that an individual of type $\Omega a \ps$ in $\mathcal B'$ has:
\begin{itemize}
 	\item \textbf{One} child of type $a\ps$ with probability $\displaystyle \frac{(1 \wedge \nu)^2 t_\nu^{2} Z_{a\ps}(\nu, t_\nu)}{1-2 \, (1 \wedge \nu) \, t_\nu Z_\ps(\nu, t_\nu)}$.
	\item \textbf{Two} children of types $\ps$ and $\Omega a\ps$ with probability $\displaystyle (1 \vee \nu) \, t_\nu Z_{\ps} (\nu,\rho_\nu)$.
	\item \textbf{One} child of type $\Omega a \ps$ with probability $\displaystyle1- (1 \vee \nu) \, t_\nu Z_{\ps} (\nu,t_\nu)- \frac{(1 \wedge \nu)^2 t_\nu^{2} Z_{a\ps}(\nu, t_\nu)}{1-2 \, (1 \wedge \nu) \, t_\nu Z_\ps(\nu, t_\nu)}$.
\end{itemize}
The second probability is taken to be larger than the branching probability in $\mathcal B$ (hence the factor $(1 \vee \nu$).
\medskip

With these choices, we can couple two branching processes $\mathcal B$  and $\mathcal B'$ so that the total number of particles in $\mathcal B'$ is larger than the total number of particles in $\mathcal B$ in the following natural way. A particle of type $t$ (ending by $\ps$) in $\mathcal B$ is projected in $\mathcal B'$ to a particle of type $\ps$ if $t=\ps$, of type $a \ps$ if $t = a \ps$ with $a \in \{ \ps, \ns \}$, or of type $\Omega a \ps$ if $t$ has length at least $3$ and ends with $a \ps$, again with $a\in \{\ps,\ns\}$.

The offspring distribution of particles of type $\ps$, $\ps\ps$ or $\ns\ps$ is equal in $\mathcal B$ and in $\mathcal B'$. Fix $a\in \{\ps,\ns\}$ and $\omega \in \{\ps,\ns\}^+$, because of the differences in the offspring distribution of a particle of type $\Omega a \ps$ and of a particle of type $\omega a \ps$, the coupled branching processes $\mathcal B$ and $\mathcal B'$ may differ in two ways: 
\begin{itemize}
\item A particle of type $\omega a \ps$ may have only one child of type $\omega a \ps$ in $\mathcal{B}$, whereas its counterpart in the coupled branching process $\mathcal B'$ has two children: one of type $\Omega a \ps$ and one of type $\ps$. 
\item A particle of type $\omega a \ps$ may have one child of type $a \ps$ whereas its counterpart in the coupled branching process $\mathcal B'$ has one child of type $\Omega a \ps$.
\end{itemize}
To prove that the total number of particles (TNOP) in $\mathcal B$ is dominated by the TNOP in  $\mathcal B'$, it is then enough to prove that the TNOP in $\mathcal{B}$ started with a single particle of type $a\ps$ is dominated by the TNOP of $\mathcal{B}'$ started with a single particle of type $\Omega a \ps$ (for the same value of $a$).
Given the definition of the offspring distributions, a branching process $\mathcal B'$ started with a particle of type $\Omega a \ps$ has to produce at least one particle of type $a \ps$ (again for the same value of $a$) to go extinct. It is hence clear that the TNOP of a branching process $\mathcal B'$ is stochastically greater when the process is started with a particle of type $\Omega a \ps$ than with a particle of type $a\ps$, which concludes the proof.
 \medskip

To prove that $\mathcal{B}'$ is critical, we write the matrix of the mean number of children of each type for $\mathcal B'$ with the ordering $(\ps, \ps \ps, \Omega \ps \ps, \ns \ps, \Omega \ns \ps)$ (all generating series $Z_{\omega}$ are evaluated at $t_\nu$):
\begin{equation} \label{eq:meanmatrix}
M =
\begin{pmatrix}
2 \nu t_\nu Z_{\ps} & \nu t_\nu\frac{Z_{\ps\ps} }{Z_\ps} & 0 & \nu t_\nu\frac{Z_{\ns\ps} }{Z_\ps} & 0\\
\nu t_\nu Z_{\ps} & 2 \nu t_\nu Z_\ps& 1 - 2 \nu t_\nu Z_\ps- \frac{\nu t_\nu}{Z_{\ps\ps} }& 0 & 0\\
\nu t_\nu Z_{\ps} & \frac{
(1 \wedge \nu)^2 \rho_\nu^{2/3} Z_{\ps \ps}
}
{1-2 \, (1 \wedge \nu) \, t_\nu Z_\ps} & 1 - \frac{
(1 \wedge \nu)^2 \rho_\nu^{2/3} Z_{\ps\ps}
}
{1-2 \, (1 \wedge \nu) \, t_\nu Z_\ps}& 0 & 0 \\
t_\nu Z_{\ps} & 0 & 0 & 2 t_\nu Z_{\ps} & 1- 2 t_\nu Z_{\ps} -  \frac{t_\nu}{Z_{\ns\ps} }\\
t_\nu Z_{\ps} & 0 & 0 & \frac{
(1 \wedge \nu)^2 \rho_\nu^{2/3} Z_{\ns \ps}
}
{1-2 \, (1 \wedge \nu) \, t_\nu Z_\ps} & 1- \frac{
(1 \wedge \nu)^2 \rho_\nu^{2/3} Z_{\ns\ps}
}
{1-2 \, (1 \wedge \nu) \, t_\nu Z_\ps}
\end{pmatrix}.
\end{equation}
To finish the proof of the proposition, we check that the spectral radius of $M$ is strictly smaller than $1$. Since we have explicit formulas for each quantity appearing in $M$, we can easily compute its spectral radius at any specified $\nu$. For $\nu=\nu_c$, we obtain $0.98985 < 1$!
\end{proof}

\bibliographystyle{plain}
\bibliography{Ising}

\bigskip

\noindent \textsc{Marie Albenque,\\ LIX, \'Ecole Polytechnique, 91120 Palaiseau – France}

\bigskip

\noindent \textsc{Laurent M\'enard,\\ Laboratoire Modal'X, UPL, Univ. Paris Nanterre, F92000 Nanterre, France}

\bigskip

\noindent \textsc{Gilles Schaeffer,\\ LIX, \'Ecole Polytechnique, 91120 Palaiseau – France}

\end{document}